\documentclass[12pt]{article}
\usepackage[margin=1in]{geometry}  
\usepackage{graphicx}              
\usepackage{amsmath}               
\usepackage{amsthm}                
\usepackage[ansinew]{inputenc}
\usepackage{array}
\usepackage{xcolor}
\usepackage{amsxtra}
\usepackage{amstext}
\usepackage{latexsym}
\usepackage{dsfont}
\usepackage[all]{xy}
\usepackage{amsmath,amsfonts,amssymb}
\usepackage{soul}

\usepackage{marginnote}
\def\hmath$#1${\texorpdfstring{{\rmfamily\textit{#1}}}{#1}}

\usepackage{enumitem}  
\usepackage{calc}
\setlist{labelindent=1pt,itemsep=.5em}
\setlist[itemize]{leftmargin=1.2cm}
\setlist[enumerate]{itemindent=0em,leftmargin=1.2cm}

\setlist[enumerate,1]{label=\upshape{(\roman*)}}
\setlist[enumerate,2]{label=\upshape{(\alph*)}}

\usepackage{cite}

\usepackage{hyperref}
\hypersetup{
colorlinks   = true,
citecolor    = black,
linkcolor =black,
urlcolor=black,
}

\hypersetup{hidelinks}

\usepackage{authblk}

\newcommand{\email}[1]{%
    \normalsize\href{mailto:#1}{\color{black}{#1} }}

\allowdisplaybreaks

\makeatletter
\newcommand{\subjclass}[2][2020]{%
  \let\@oldtitle\@title%
  \gdef\@title{\@oldtitle\footnotetext{#1 \emph{Mathematics subject classification}: #2}}%
}
\newcommand{\keywords}[1]{%
  \let\@@oldtitle\@title%
  \gdef\@title{\@@oldtitle\footnotetext{\emph{Keywords}: #1}}%
}
\makeatother

\newtheorem{thm}{Theorem}[section]
\newtheorem{cor}[thm]{Corollary}
\newtheorem{lem}[thm]{Lemma}
\newtheorem{prop}[thm]{Proposition}
\theoremstyle{definition}
\newtheorem{defn}[thm]{Definition}
\theoremstyle{remark}
\newtheorem{rmk}[thm]{Remark}
\theoremstyle{remark}

\newtheorem{ex}[thm]{Example}

\numberwithin{equation}{section}

\title{Averaging operators on $q$-deformed Witt and $q$-deformed $W(2,2)$ algebras}

\author[1,2]{Ismail Laraiedh}
\author[3]{Sergei Silvestrov}
\affil[1]{\Affilfont Departement of Mathematics, Faculty of Sciences,
\authorcr \Affilfont Sfax University, Box 1171, 3000 Sfax, Tunisia
\authorcr \Affilfont
\email{ismail.laraiedh@gmail.com}}
\affil[2]{Departement of Mathematics,
\authorcr \Affilfont College of Sciences and Humanities Al Quwaiiyah,
\authorcr \Affilfont Shaqra University, Kingdom of Saudi Arabia
\authorcr \Affilfont
\email{ismail.laraiedh@su.edu.sa}}
\affil[3]{\Affilfont Division of Mathematics and Physics,
\authorcr \Affilfont School of Education, Culture and Communication,
\authorcr \Affilfont M\"{a}lardalen University, Box 883, 72123 V{\"a}ster{\aa}s, Sweden
\authorcr \Affilfont
\email{sergei.silvestrov@mdu.se}}


\date{}

\begin{document}
\maketitle
\begin{abstract}
The aim of this paper is to give some constructions results of
averaging operators on Hom-Lie algebras. The homogeneous averaging operators on $q$-deformed Witt and $q$-deformed $W(2,2)$ Hom-algebras are classified. As applications, the induced Hom-Leibniz algebra structures are obtained and their multiplicativity conditions are also given.
\end{abstract}
\section{Introduction}
The investigations of various quantum deformations or $q$-deformations of Lie algebras began a period of rapid expansion in 1980's stimulated by introduction of quantum groups motivated by applications to the quantum Yang-Baxter equation, quantum inverse scattering methods and constructions of the quantum deformations of universal enveloping algebras of semi-simple Lie algebras. Various $q$-deformed Lie algebras have appeared in physical contexts such as string theory, vertex models in conformal field theory, quantum mechanics and quantum field theory in the context of deformations of infinite-dimensional algebras, primarily the Heisenberg algebras, oscillator algebras and Witt and Virasoro algebras. In \cite{AizawaSaito,ChaiElinPop,ChaiIsLukPopPresn,ChaiKuLuk,ChaiPopPres,CurtrZachos1,DamKu,DaskaloyannisGendefVir,Hu,Kassel92,LiuKQuantumCentExt,LiuKQCharQuantWittAlg,LiuKQPhDthesis},
it was in particular discovered that in these $q$-deformations of Witt and Visaroro algebras and some related algebras, some interesting $q$-deformations of Jacobi identities, extending Jacobi identity for Lie algebras, are satisfied. This has been one of the initial motivations for the development of general quasi-deformations and discretizations of Lie algebras of vector fields using more general $\sigma$-derivations (twisted derivations) in \cite{HartwigLarssonSilvestrov:defLiealgsderiv}.

Hom-Lie algebras and more general quasi-Hom-Lie algebras were introduced first by Hartwig, Larsson and Silvestrov \cite{HartwigLarssonSilvestrov:defLiealgsderiv}, where the general quasi-deformations and discretizations of Lie algebras of vector fields using more general $\sigma$-derivations (twisted derivations) and a general method for construction of deformations of Witt and Virasoro type algebras based on twisted derivations have been developed, initially motivated by the $q$-deformed Jacobi identities observed for the $q$-deformed algebras in physics, along with $q$-deformed versions of homological algebra and discrete modifications of differential calculi. Hom-Lie algebras, Hom-Lie superalgebras, Hom-Lie color algebras and more general quasi-Lie algebras and color quasi-Lie algebras where introduced first in \cite{LarssonSilv2005:QuasiLieAlg,LarssonSilv:GradedquasiLiealg,SigSilv:CzechJP2006:GradedquasiLiealgWitt}. Quasi-Lie algebras and color quasi-Lie algebras encompass within the same algebraic framework the quasi-deformations and discretizations of Lie algebras of vector fields by $\sigma$-derivations obeying twisted Leibniz rule, and the well-known generalizations of Lie algebras such as color Lie algebras, the natural generalizations of Lie algebras and Lie superalgebras. In quasi-Lie algebras, the skew-symmetry and the Jacobi identity are twisted by deforming twisting linear maps, with the Jacobi identity in quasi-Lie and quasi-Hom-Lie algebras in general containing six twisted triple bracket terms. In Hom-Lie algebras, the bilinear product satisfies the non-twisted skew-symmetry property as in Lie algebras, and the Hom-Lie algebras Jacobi identity has three terms twisted by a single linear map, reducing to the Lie algebras Jacobi identity when the twisting linear map is the identity map. Hom-Lie admissible algebras have been considered first in \cite{ms:homstructure}, where in particular the Hom-associative algebras have been introduced and shown to be Hom-Lie admissible, that is leading to Hom-Lie algebras using commutator map as new product, and in this sense constituting a natural generalization of associative algebras as Lie admissible algebras. Since the pioneering works \cite{HartwigLarssonSilvestrov:defLiealgsderiv,LarssonSilvJA2005:QuasiHomLieCentExt2cocyid,LarssonSilv:GradedquasiLiealg,LarssonSilv2005:QuasiLieAlg,LarssonSilv:QuasidefSl2,ms:homstructure}, Hom-algebra structures expanded into a popular area with increasing number of publications in various directions. Hom-algebra structures of a given type include their classical counterparts and open broad possibilities for deformations, Hom-algebra extensions of cohomological structures and representations, formal deformations of Hom-associative and Hom-Lie algebras, Hom-Lie admissible Hom-coalgebras, Hom-coalgebras, Hom-Hopf algebras \cite{AmmarEjbehiMakhlouf:homdeformation,BenMakh:Hombiliform,ElchingerLundMakhSilv:BracktausigmaderivWittVir,LarssonSilvJA2005:QuasiHomLieCentExt2cocyid,LarssonSilvestrovGLTMPBSpr2009:GenNComplTwistDer,MakhSil:HomHopf,MakhSilv:HomAlgHomCoalg,MakhSilv:HomDeform,
Sheng:homrep,ShengBai2014:homLiebialg,Yau:HomolHom,Yau:EnvLieAlg}.
Hom-Lie algebras, Hom-Lie superalgebras and color Hom-Lie algebras and their $n$-ary generalizations have been further investigated in various aspects for example in \cite{AbdaouiAmmarMakhloufCohhomLiecolalg2015,belhsine,AbramovSilvestrov:3homLiealgsigmaderivINvol,AmmarEjbehiMakhlouf:homdeformation,AmmarMabroukMakhloufCohomnaryHNLalg2011,AmmarMakhloufHomLieSupAlg2010,AmmarMakhloufSaadaoui2013:CohlgHomLiesupqdefWittSup,AmmarMakhloufSilv:TernaryqVirasoroHomNambuLie,ArmakanFarhangdoost:IJGMMP,ArmakanSilv:envelalgcertaintypescolorHomLie,ArmakanSilvFarh:envelopalgcolhomLiealg,ArmakanSilvFarh:exthomLiecoloralg,ArmakanSilv:NondegKillingformsHomLiesuperalg,
akms:ternary,ams:ternary,ArnlindMakhloufSilvnaryHomLieNambuJMP2011,Bakayoko2014:ModulescolorHomPoisson,BakayokoDialo2015:genHomalgebrastr,BakyokoSilvestrov:Homleftsymmetriccolordialgebras,BakyokoSilvestrov:MultiplicnHomLiecoloralg,BakayokoToure2019:genHomalgebrastr,
BenHassineMabroukNcib:ConstrMultiplicnaryhomNambualg,BenHassineChtiouiMabroukNcib:Strcohom3LieRinehartsuperalg,BenMakh:Hombiliform,CaoChen2012:SplitregularhomLiecoloralg,GuanChenSun:HomLieSuperalgebras,KitouniMakhloufSilvestrov,kms:solvnilpnhomlie2020,LarssonSigSilvJGLTA2008:QuasiLiedefFttN,MabroukNcibSilvestrov2020:GenDerRotaBaxterOpsnaryHomNambuSuperalgs,
ms:homstructure,MakhSilv:HomDeform,MakhSil:HomHopf,MakhSilv:HomAlgHomCoalg,Makhlouf2010:ParadigmnonassHomalgHomsuper,MandalMishra:HomGerstenhaberHomLiealgebroids,MishraSilvestrov:SpringerAAS2020HomGerstenhalgsHomLiealgds,RichardSilvJA2008:quasiLiesigderCtpm1,RichardSilvestrovGLTMPBSpr2009:QuasiLieHomLiesigmaderiv,
SigSilv:CzechJP2006:GradedquasiLiealgWitt,Sheng:homrep,ShengBai2014:homLiebialg,ShengChen2013:HomLie2algebras,ShengXiong:LMLA2015:OnHomLiealg,SigSilv:GLTbdSpringer2009,SilvestrovParadigmQLieQhomLie2007,
Yau2009:HomYangBaxterHomLiequasitring,Yau:EnvLieAlg,Yau:HomolHom,Yau:HomBial,Yuan2012:HomLiecoloralgstr,ZhouChenMa:GenDerHomLiesuper,ZhouChenMa:GenDerLieTripSyst,ZhouNiuChen:GhomDerivation,ZhouZhaoZhang:GenDerHomLeibnizalg}.

In the 1930s, the notion of averaging operator was explicitly defined by Kolmogoroff and Kamp\'{e} de F\'{e}riet~\cite{KF,Mil}.  Then G. Birkhoff \cite{Bi} continued its study and showed that a positive bounded projection in the Banach algebra $C(X)$, the algebra of scalar valued continuous functions on a compact Hausdorff space $X$, onto a fixed range space is an idempotent averaging operator. In 1954, S.~T.~C. Moy~\cite{STC} made the connection between averaging operators and conditional expectation. Furthermore, she studied the relationship between integration theory and averaging operators in turbulence theory and probability. Then her results were extended by G.~C. Rota \cite{R1}. During the same period, the idempotent averaging operators on $C_{\infty}(X)$, the algebra of all real valued continuous functions on a locally compact Hausdorff space $X$ that vanish at the infinity, were characterized by J.~L. Kelley \cite{Kel}.

In this century, while averaging operators continued to find many applications in its traditional areas of analysis and applied areas~\cite{Fe}, their algebraic study has been deepened and generalized.
J.~L. Loday~\cite{Lo} defined the diassociative algebra as the enveloping algebra of the Leibniz algebra by analogy with the associative algebra as the enveloping algebra of the Lie algebra.
More precisely, an averaging operator on an algebra $A$ over a field $\mathbb{K}$ is a linear map $P : A \rightarrow A$ satisfying
the averaging relation:
$$P(x)P(y)=P(P(x)y)=P(xP(y)).$$
M. Aguiar in \cite{Ag} showed that a diassociative algebra can be derived from an averaging associative algebra by defining two new operations $x \dashv y := x P(y)$ and $x \vdash y: = P(x)y$. An analogue process gives a Leibniz algebra from an averaging Lie algebra by defining a new operation $\{x,y\} := [P(x),y]$ and derives a (left) permutative algebra from an averaging commutative associative algebra. In general, an averaging operator was defined on any binary operad and this kind of process was systematically studied in~\cite{PBGN} by relating the averaging actions to a special construction of binary operads called duplicators~\cite{GK2,PBGN2}.

The purpose of this paper is to give some constructions results of
averaging operators on Hom-Lie algebras and to classify the homogeneous averaging operators on $q$-deformed Witt and $q$-deformed $W(2,2)$ algebras. Then the induced Leibniz algebra structures are obtained.
Section \ref{section2} contains some necessary important basic notions, notations and examples on $\mathbb{Z}$-graded Hom-Lie algebras which will be used in
next sections and we study the multiplicativity conditions of $q$-deformed Witt and $q$-deformed $W(2,2)$ Hom algebras. Next, we present some useful methods for constructions of averaging operator on  Hom-Lie algebras. In section \ref{section3}, we classify the  homogeneous averaging operators on the $q$-deformed
Witt Hom-algebra $\mathcal{V}^{q}$ and we give the induced Hom-Leibniz algebras
from the averaging operators on the $q$-deformed Witt Hom-algebra $\mathcal{V}^{q}$. In section \ref{section4}, we classify the  homogeneous averaging operators on the $q$-deformed $W(2,2)$ Hom-algebra $\mathcal{W}^{q}$. Also, we give the induced Hom-Leibniz algebras
from the averaging operators on the $q$-deformed $W(2,2)$ Hom-algebra $\mathcal{W}^{q}$.

\section{Constructions of averaging operators on Hom-Lie algebras}\label{section2}
In this section, firstly, we review some important basic notions, notations and examples on $\mathbb{Z}$-graded Hom-Lie algebras which will be used in next sections. Then, we present some useful methods for constructions of averaging operator on Hom Lie algebras.

In this article, all linear spaces are over a field $\mathbb{K}$ of characteristic zero. A linear operator $T:A\mapsto A$ on a $\mathbb{Z}$-graded linear space $A =\bigoplus_{j\in\mathbb{Z}} V_j,$ is said to respect the grading of the linear space $A$ if for any $i\in\mathbb{Z}$ there exists $j\in\mathbb{Z}$ such that $T(A_i)\subseteq A_j$.
The linear operator respecting grading is said to be homogeneous of degree $\deg T \in\mathbb{Z}$ if
$
T(A_i)\subseteq A_{i+\deg T}
$
for all $i\in\mathbb{Z}$, and $T$ is said to be even if $\deg T=0$, that is
$T(A_i)\subseteq A_{i}$ for all for all $i\in\mathbb{Z}$.

\subsection{Hom-algebras, Hom-Lie algebras and multiplicativity}
Hom-algebras in general are triples $(A,[\cdot,\cdot],\alpha)$ consisting of a linear space $A$, bilinear product $[\cdot,\cdot]:A\times A\mapsto A$ and a linear map (linear space homomorphism) $\alpha:A\mapsto A$.
If, moreover, the linear map $\alpha: A\rightarrow A$ is an algebra endomorphism, meaning that it satisfies for all $x,y\in A$ the multiplicativity property
\begin{eqnarray} \label{multipl:homleibnizhomlie}
\alpha([x, y])=[\alpha(x),\alpha(y)],
\end{eqnarray}
then the Hom-Lie algebra is called {\it multiplicative}.
Within specific classes of Hom-algebras, it is important to characterize multiplicative and non-multiplicative Hom-algebras belonging to the class.

A Hom-algebra, $(A,[\cdot,\cdot],\alpha)$ is said to be {\it $\mathbb{Z}$-graded} if the linear space
$A$ is $\mathbb{Z}$-graded,
\begin{equation*}
A =\bigoplus_{j\in\mathbb{Z}}A_j,
\end{equation*}
the bilinear product $[\cdot,\cdot]$ is $\mathbb{Z}$-graded, that is, for all $m,n\in\mathbb{Z}$,
\begin{equation*}
[A _m, A_n]\subseteq A_{m+n},
\end{equation*}
and the linear operator
$\alpha$ respects the $\mathbb{Z}$-grading of the linear space $A$, that is, for any $i\in\mathbb{Z}$ there exists $j\in\mathbb{Z}$ such that $\alpha(A_i)\subseteq A_j$.

In any $\mathbb{Z}$-graded Hom-algebra $(A =\bigoplus_{j\in\mathbb{Z}}A_j,[\cdot,\cdot],\alpha)$, and the following inclusions hold:
\begin{align*}
         & [\alpha(A_m),\alpha(A_n)] \subseteq [A_{m+k},A_{n+k}]\subseteq A_{m+n+2k},\\
         &\alpha([A_m,A_n]) \subseteq \alpha(A_{m+n})\subseteq A_{m+n+k},\\
         & [\alpha(A_m),\alpha(A_n)] \cap \alpha([A_m,A_n]) \subseteq
         A_{m+n+2k} \cap A_{m+n+k} =
         \left\{\begin{array}{ll}
         A_{m+n},& \ \text{if} \ k=0 \\
         0,& \ \text{if} \ k\neq 0
         \end{array},
         \right. \\
         &
         \begin{array}{l}
        \ker([\cdot,\cdot])\cap
         ((\ker(\alpha)\times A) \cup (A\times \ker(\alpha)))\\
        \quad \subseteq M_{A,[\cdot,\cdot],\alpha}=
        \{(x,y)\in A\times A \mid [\alpha(x),\alpha(y)]=\alpha([x,y])
          \} \\
         \end{array}
     \end{align*}

These inclusions directly yield the following handy conditions for checking whether $\mathbb{Z}$-graded Hom-algebras are
multiplicative or non-multiplicative, based on an interaction between
the bilinear product $[\cdot,\cdot]$, the twisting map $\alpha$, the $\mathbb{Z}$-grading of $A$ and elements of its homogeneous subspaces $A_j, j\in\mathbb{Z}$ in the $\mathbb{Z}$-grading direct decomposition.
\begin{thm}
\label{thm:multnonmultcnds}
Let $(A=\displaystyle{\bigoplus_{n\in\mathbb{Z}}} A_n,[\cdot,\cdot],\alpha)$  be a $\mathbb{Z}$-graded
Hom-algebra where $\alpha$ is a linear operator homogeneous of degree $\deg \alpha = k\in\mathbb{Z}$.
\begin{enumerate}
\item
\label{thm:i:multnonmultcnds}
The Hom-algebra $(A,[\cdot,\cdot],\alpha)$  is not multiplicative, if and only if
\begin{equation*}
\exists\ m,n\in \mathbb{Z}, x_m\in A_m, x_n\in A_n:
[\alpha(x_m),\alpha(x_n)]\neq \alpha([x_m,x_n]).
\end{equation*}
The Hom-algebra $(A,[\cdot,\cdot],\alpha)$ is multiplicative if and only if
\begin{equation*}
\forall \ m,n\in \mathbb{Z}, x_m\in A_m, x_n\in A_n:
[\alpha(x_m),\alpha(x_n)]=\alpha([x_m,x_n]).
\end{equation*}
\item \label{thm:ii:multnonmultcnds}
$(A,[\cdot,\cdot],\alpha)$  is not multiplicative, if and only if the strict inclusion takes place
\begin{equation*}
\exists\ m,n\in \mathbb{Z}:
\left\{(x_m,x_n)\in A_m\times A_n \mid
[\alpha(x_m),\alpha(x_n)]=\alpha([x_m,x_n]) \right\}
      \subsetneq A_m\times A_n,
\end{equation*}
or equivalently if and and only if
\begin{equation*}
\begin{array}{l}
\exists\ m,n\in \mathbb{Z}: \\
 A_m\times A_n
\setminus  \{(x_m,x_n)\in A_m\times A_n \mid [\alpha(x_m),\alpha(x_n)]=\alpha([x_m,x_n])
         \} \neq \emptyset
\end{array}
\end{equation*}
\item \label{thm:iii:multnonmultcnds}
\label{thm:item:mult}
If $(A,[\cdot,\cdot],\alpha)$ is multiplicative, then one of the following alternatives holds:
\begin{enumerate} 
\item  \label{thm:iiia:multnonmultcnds}
linear operator $\alpha$ is even, that is homogeneous of degree $k=0$;
\item \label{thm:iiib:multnonmultcnds}
linear operator $\alpha$ is homogeneous of degree $k\neq 0$ and
\begin{equation*}
\forall \ m,n\in \mathbb{Z}: \ [\alpha(A_m),\alpha(A_n)]\times\alpha([A_m,A_n])=\{0\}\times\{0\}=\{(0,0)\},
\end{equation*}
by linearity of $\alpha$ and bilinearity of $[\cdot,\cdot]$ equivalent to $[\alpha(\cdot),\alpha(\cdot)]=\alpha([\cdot,\cdot])=0.$
\end{enumerate}
\item  \label{thm:iv:multnonmultcnds}
If $k\neq 0$, then
$(A,[\cdot,\cdot],\alpha)$ is not multiplicative if and only if
\begin{equation*}
\exists\ m,n\in \mathbb{Z}: [\alpha(A_m),\alpha(A_n)]\times\alpha([A_m,A_n])\neq\{0\}\times\{0\}=\{(0,0)\}.
\end{equation*}
If $k\neq 0$, then
$(A,[\cdot,\cdot],\alpha)$ is multiplicative if and only if
\begin{equation*}
\forall\ m,n\in \mathbb{Z}: [\alpha(A_m),\alpha(A_n)]\times\alpha([A_m,A_n])=\{0\}\times\{0\}=\{(0,0)\},
\end{equation*}
that is $
\forall\ m,n\in \mathbb{Z}: [\alpha(A_m),\alpha(A_n)]=\alpha([A_m,A_n])=\{0\},
$
which is the same as
\begin{equation*}
[A,A]\subseteq \ker(\alpha),\quad \alpha(A)\subseteq \ker([\cdot,\cdot])=\{(x,y)\in A\times A \mid [x,y]=0\}.
\end{equation*}
or equivalently, for elements of the homogeneous subspaces, if and only if
\begin{equation*}
\forall \ m,n\in \mathbb{Z}, x_m\in A_m, x_n\in A_n:
[\alpha(x_m),\alpha(x_n)]=\alpha([x_m,x_n])=0.
\end{equation*}
\end{enumerate}
\end{thm}

If $\dim A_m=1$ for all $m\in \mathbb{Z}$ and $\{x_m\in A_m, m\in\mathbb{Z}\}$ is a homogeneous basis of the $\mathbb{Z}$-graded linear space $A= \bigoplus\limits_{m\in \mathbb{Z}} A_m$, then
for all $m,n\in \mathbb{Z}$,
\begin{align*}
& \alpha(x_{m})=\alpha_{m+k,m} x_{m+k}, \quad \text{for some unique}\ \alpha_{m+k,m} \in \mathbb{K} \\
& \alpha(x_{m+n}) = \alpha_{m+n+k,m+n} x_{n+m+k}, \quad \text{for some unique}\ \alpha_{m+n+k,m+n} \in \mathbb{K} \\
& \left[x_m,x_n\right] = c^{m+n}_{m,n}x_{m+n}, \quad \text{for some unique}\ c^{m+n}_{m,n} \in \mathbb{K}\\
& \left[\alpha(x_m),\alpha(x_n)\right]= \alpha_{m+k,m}\alpha_{n+k,n} c^{m+n+2k}_{m+k,n+k}
x_{m+n+2k}, \\
& \alpha([x_m,x_n])=  c^{m+n}_{m,n}\alpha(x_{m+n})= c^{m+n}_{m,n}\alpha_{m+n+k,m+n}x_{m+n+k}.
\end{align*}

\begin{cor} \label{cor:Hom-algebramultiplicativity}
Let $\mathcal{A}=(A=\displaystyle{\bigoplus_{n\in\mathbb{Z}}}A_n,[\cdot,\cdot],\alpha)$  be a $\mathbb{Z}$-graded
Hom-algebra where $\alpha$ is a homogeneous linear operator of degree $\deg \alpha = k\in\mathbb{Z}$.
If $\dim A_m=1$ for all $m\in \mathbb{Z}$ and $\{x_m\in A_m, m\in\mathbb{Z}\}$ is a homogeneous basis of the $\mathbb{Z}$-graded linear space $A= \oplus_{m\in \mathbb{Z}} A_m$, then
for all $m,n\in \mathbb{Z}$,
\begin{enumerate}
\item \label{item1:corollary}If $\deg \alpha =k\neq 0$, then $\mathcal{A}$ is multiplicative if and only if
for all $m, n \in \mathbb{Z}$,
\begin{equation*}
\alpha_{m+k,m}\alpha_{n+k,n} c^{m+n+2k}_{m+k,n+k} = c^{m+n}_{m,n}\alpha_{m+n+k,m+n} = 0,
\end{equation*}
which is equivalent to
$
\left\{
\begin{array}{l}
\alpha_{m+n+k,m+n} = 0,\quad \text{if}\quad c^{m+n}_{m,n}\neq 0 \\
\alpha_{m+k,m}=0 \ \text{or}\ \alpha_{n+k,n}=0, \quad \text{if}\quad c^{m+n+2k}_{m+k,n+k}\neq 0
\end{array}
\right.
$
\item\label{item2:corollary} If $\deg \alpha =k=0$, then $\mathcal{A}$ is multiplicative if and only if, for all $m, n \in \mathbb{Z}$,
\begin{equation*}
\alpha_{m,m}\alpha_{n,n} c^{m+n}_{m,n}=c^{m+n}_{m,n}\alpha_{m+n,m+n},
\end{equation*}
that is if and only if, for all $m, n \in \mathbb{Z}$,
\begin{equation*}
c^{m+n}_{m,n}(\alpha_{m,m}\alpha_{n,n}-\alpha_{m+n,m+n})=0,
\end{equation*}
or equivalently if and only if, for all $m, n \in \mathbb{Z}$,
\begin{equation*}
\alpha_{m,m}\alpha_{n,n}=\alpha_{m+n,m+n},\quad \text{if}\quad c^{m+n}_{m,n}\neq 0.
\end{equation*}
\end{enumerate}
\end{cor}

\begin{defn}[\cite{HartwigLarssonSilvestrov:defLiealgsderiv,ms:homstructure}]
\label{def:homliealg}
Hom-Lie algebras are Hom-algebras $( A, [\cdot,\cdot], \alpha)$ consisting of a linear space $A$ over a field  $\mathbb{K}$, a bilinear map $[\cdot,\cdot]$: $ A\times A\rightarrow  A$ and a linear map $\alpha:  A \rightarrow A$, satisfying for all $x,y,z \in A$,
\begin{align}
\label{skewsymmetry}
[x,y]&=-[y,x] &\quad \quad \text{(Skew-symmetry identity)}\\
\label{homliejacobiidentity}
[\alpha(x),[y, z]] &+[\alpha(y),[z,x]]+[\alpha(z),[x,y]] = 0. &\quad \quad \text{(Hom-Jacobi identity)}
\end{align}
\end{defn}

\begin{defn}[\cite{LarssonSilv2005:QuasiLieAlg,ms:homstructure}]
\label{def:homleibnizalg}
Hom-Leibniz algebras are Hom-algebras $( A, [\cdot,\cdot], \alpha)$ consisting of a linear space $ A$ over a field $\mathbb{K}$, a bilinear map $[\cdot,\cdot]$: $ A\times A\rightarrow  A$ and a linear map $\alpha:  A\rightarrow A$ satisfying for all $x,y,z \in A$,
\begin{equation}\label{Leibnizidentity}
[\alpha(x),[y, z]]=[[x,y],\alpha(z)]+[\alpha(y),[x,z]]. \quad \quad \text{(Hom-Leibniz identity)}
\end{equation}
When, moreover, the linear map $\alpha: A\rightarrow A$ satisfies multiplicativity \eqref{multipl:homleibnizhomlie}, that is when $\alpha$ is an algebra endomorphism,
the Hom-Leibniz algebra $(A, [\cdot,\cdot], \alpha)$ is called multiplicative.
\end{defn}

\begin{rmk}
Skewsymmetric Hom-algebras are Hom-algebras satisfying the skewsymmetry axiom \eqref{skewsymmetry}, and hence the Hom-Lie algebras form a special subclass of skewsymmetric Hom-algebras where moreover the Hom-Jacobi identity \eqref{homliejacobiidentity} holds.
In skewsymmetric algebras however there is no requirement of any relations between the linear operation $\alpha$ and bilinear operation $[\cdot,\cdot]$. In this sence, the skewsymmetric Hom-algebras can be seen and studied just as arbitrary pairs of skewsymmetric algebras and linear operators on them. However, this is not the case in Hom-Lie or Hom-Leibniz algebras where the linear and bilinear operations are dependent via Hom-Jacobi and Hom-Leibniz identities in nontrivial ways.
\end{rmk}
\begin{rmk}
Every skewsymmetric Hom-Leibniz algebra is a Hom-Lie algebra, every Hom-Lie algebra is a skewsymmetric Leibniz algebra, but not every Hom-Leibniz algebra is skewsymmetric, and thus Hom-Lie algebras  as a class of Hom-algebras coincides with the intersection of the class of Hom-Leibniz algebras and the class of skewsymmetric algebras, which is moroever properly included in each of classes.
\end{rmk}

\begin{ex}\label{ex1}
For $q\in\mathbb{K}\setminus\{0\}$ and $n\in\mathbb{Z}$, the $q$-numbers $\{n\}$ defined by
$$
\{n\}=
\left\{\begin{array}{ll}
\frac{1-q^n}{1-q}, & \text{for} \  q\neq 1 \\
n, & \text{for} \ q=1
\end{array}
\right.
$$
have the following properties
\begin{equation}
\begin{array}{c}
\{m+1\}=1+q\{m\}=\{m\}+q^m,\ \{m+n\}=\{m\}+q^m\{n\},\ q^m\{-m\}=-\{m\},\\
\{m\}=0 \ \Leftrightarrow \ q^{m}=1.
\end{array}
\end{equation}
The linear space $\mathcal {V}^{q}$ with a basis $\{L_n|n\in\mathbb{Z}\}$ equipped with the bilinear operation $[\cdot,\cdot]$ and a linear map $\alpha$ on $\mathcal {V}^{q}$
on the basis, for all $m, n\in\mathbb{Z}$, by
\begin{align}
\label{witt bracket}
&\left[L_m,L_n\right]=(\{m\}-\{n\})L_{m+n}, \quad  \\
&\alpha(L_n)=(1+q^n)L_n. \nonumber
\end{align}
Then, $(\mathcal {V}^{q}, [\cdot,\cdot], \alpha)$
is a Hom-Lie algebra \cite{HartwigLarssonSilvestrov:defLiealgsderiv,SigSilv:CzechJP2006:GradedquasiLiealgWitt,SigSilv:GLTbdSpringer2009}, called the \textit{$q$-deformed Witt Hom-Lie algebra} or \textit{$q$-Witt Hom-Lie algebra}.
There is a natural $\mathbb{Z}$-grading on $\mathcal{V}^{q}$,
$$
\mathcal{V}^{q}=\bigoplus_{n\in\mathbb{Z}}\mathcal{V}^{q}_n, \ \mathcal{V}^{q}_n=\mathbb{K}L_n,\ n\in \mathbb{Z}.$$
\end{ex}

\begin{ex}
\label{ex1modified:degk}
If, in Example \ref{ex1}, the linear operator $\alpha$, homogeneous of degree $k$, is defined for all $n\in\mathbb{Z}$, by
$\alpha=\alpha_k(L_n)=(1+q^{n-k})L_{n+k},$ then $(\mathcal {V}^{q}, [\cdot,\cdot], \alpha_k)$ are  $\mathbb{Z}$-graded Hom-Lie algebras for all $k\in \mathbb{Z}$.
\end{ex}

\begin{ex}\label{ex2}
For $q\neq 0$ and  $n\in\mathbb{Z}$, let $[n]$ denote
the $q$-number
\begin{equation*} 
[n]=[n]_q=\left\{\begin{array}{l}
\frac{q^n-q^{-n}}{q-q^{-1}}, \ \text{if}\ q\neq \pm 1\\
n,  \ \text{if}\ q=1 \\
(-1)^{n-1} n=(-1)^{n+1}n=-(-1)^n n, \ \text{if}\ q=-1.  \end{array} \right.
\end{equation*}
Note that these $q$-numbers are invariant under transformation replacing $q$ by $q^{-1}$, and satisfy for all $m,n\in \mathbb{Z}$,
\begin{equation*}
\label{qnumbersymproperties}
\left\{\begin{array}{l}
[-n]=-[n],\ q^{n}[m]-q^{m}[n]=[m-n],\   q^{-n}[m]+q^{m}[n]=[m+n], \ \text{for all}\  q\in\mathbb{K}\setminus\{0\} \\
\left[n\right]=0 \ \Leftrightarrow \ q^{2n}=1,  \ \text{for all}\ q\neq \pm 1 \\
\left[n\right]=n=0 \ \Rightarrow \ n=0, q^{2n}=1^{0}=1, \ \text{for}\ q= 1 \\
\left[n\right]=(-1)^{n-1} n=0 \ \Rightarrow \ n=0, q^{2n}=(-1)^{0}=1, \ \text{for}\ q=-1.
\end{array}\right.
\end{equation*}
Note that if $q=\pm 1$, then
$q^{2n}=1$ for all $n\in\mathbb{Z}$, while
$\left[n\right]=\left\{\begin{array}{l}
n,  \ \text{if}\ q=1 \\
(-1)^{n-1} n, \ \text{if}\ q=-1\\
\end{array} \right. =0  $ only for $n=0$.

Let $\mathcal{W}^{q}$ be a linear space with basis $\{L_n,W_n|n\in\mathbb{Z}\}$, and a bilinear operation on $\mathcal{W}^{q}$ is defined on the basis, for all $m,n\in\mathbb{Z}$, by
 \begin{eqnarray}\label{H-L1}
            [L_m,L_n]=[m-n]L_{m+n},\ \
            {[L_m,W_n]=[m-n]W_{m+n}},
        \end{eqnarray}
        and with other brackets obtained by skew-symmetry or equal to 0.
        The linear map $\alpha$ on $\mathcal {W}^q$ is defined, for all $n\in\mathbb{Z}$, by
        \begin{equation*} 
            \alpha(L_n)=(q^n+q^{-n})L_n,\ \  \alpha(W_n)=(q^n+q^{-n})W_n.
 \end{equation*}
It was proved in \cite{YY} that the triple ($\mathcal{W}^q, [\cdot,\cdot], \alpha$)
forms a Hom-Lie algebra, which is called the \textit{$q$-deformed $W(2,2)$ Hom-Lie algebra}.
By defining ${\rm deg}(L_n)={\rm deg}(W_n)=n$, we obtain that $\mathcal{W}^q$ is $\mathbb{Z}$-graded Hom-Lie algebra, namely  $\mathcal{W}^q=\bigoplus_{n\in\mathbb{Z}}\mathcal{W}^q_n$ with $\mathcal{W}^q_n={\rm span}_{\mathbb{K}}\{L_n,W_n\}.$
Note that $\mathcal{W}^q$ is not multiplicative since $\alpha$ is not a homomorphism of Hom-Lie algebras.
\end{ex}

\begin{ex}\label{ex2modified:degk}
In Example \ref{ex2}, if the homogeneous linear operator of degree $k$ is defined by
$ \beta(L_n)=(q^{n-k}+q^{k-n})L_{n+k}$,  $\beta(W_n)=(q^{n-k}+q^{k-n})W_{n+k}$ for all $n\in\mathbb{Z},$ then
$(\mathcal {W}^q, [\cdot,\cdot], \beta)$ is a $\mathbb{Z}$-graded Hom-Lie algebra.
\end{ex}

\begin{prop}
For any $k\in \mathbb{Z}$, the Hom-Lie algebra $(\mathcal {V}^{q}, [\cdot,\cdot], \alpha_k)$ is multiplicative if and only if $k=0$ and $q=-1$.
\end{prop}
\begin{proof}
If $k=\deg\alpha_k = \deg\alpha_0= 0$, then
by Corollary \ref{cor:Hom-algebramultiplicativity} \ref{item2:corollary},
\begin{align*}
& (\mathcal {V}^{q}, [\cdot,\cdot], \alpha)\quad \text{is multiplicative} \ \Leftrightarrow \\
&\forall\ m,n\in\mathbb{Z}:\ (\{m\}-\{n\})\big((1+q^{m})(1+q^{n})-(1+q^{m+n})\big)=0 \ \Leftrightarrow \\
&\forall\ m,n\in\mathbb{Z}:\ (\{m\}-\{n\})(q^n+q^m)=0 \ \Leftrightarrow \\
& \forall\ m,n\in\mathbb{Z}:\
\left\{\begin{array}{ll}
(q^n-q^m)(q^n+q^m)=0, & \text{for} \  q\neq 1 \\
(m-n)(q^n+q^m)=(m-n)\cdot 2=0, & \text{for} \ q=1
\end{array}
\right.
 \ \Leftrightarrow \\
&\forall\ m,n\in\mathbb{Z}:\
\left\{\begin{array}{ll}
q^{2(n-m)}=1, & \text{for} \  q\neq 1 \\
m-n=0, & \text{for} \ q=1
\end{array}\right.
 \ \Leftrightarrow \\
& q\neq 1, \forall\ p \in \mathbb{Z}:\ q^{2p}=1 \Leftrightarrow
q\neq 1, q^{2}=1 \ \text{(for $p=1$)} \Leftrightarrow \ q=-1.
\end{align*}
If $k=\deg\alpha_k \neq 0$, then by Corollary \ref{cor:Hom-algebramultiplicativity} \ref{item1:corollary},
\begin{align}
& (\mathcal {V}^{q}, [\cdot,\cdot], \beta)\quad \text{is multiplicative} \ \Leftrightarrow \nonumber \\
&\forall\ m,n\in\mathbb{Z}:\ \left\{
\begin{array}{l}
(\{m\}-\{n\})(1+q^{m+n-k})= 0; \\
(1+q^{m-k})(1+q^{n-k})(\{m+k\}-\{n+k\})=0,
\end{array}\right. \Leftrightarrow \nonumber \\
&\forall\  m,n\in\mathbb{Z}:\
\left\{
\begin{array}{l}
\left\{\begin{array}{l}
(q^n-q^m)(1+q^{m+n-k})= 0,\ \text{if}\ q\neq 1 \\
(m-n)\cdot 2 =0,\ \text{if}\ q= 1
\end{array}\right.;
\\
\left\{\begin{array}{l}
 (1+q^{m-k})(1+q^{n-k})(q^{n+k}-q^{m+k})=0,\ \text{if}\ q\neq 1 \\
4\cdot (m-n)=0, \text{if}\ q=1
\end{array}\right.
\end{array}
\right. \Leftrightarrow  \nonumber  \\
&q\neq 1 \ \text{and}\ \forall\  m,n\in\mathbb{Z}:\
\left\{
\begin{array}{lll}
& q^{n}=q^{m} \ \text{or} \ q^{n+m-k}=-1; \\
& q^{m-k}=-1 \ \text{or} \ q^{n-k}=-1 \ \text{or} \ q^{n}=q^{m}
\end{array}
\right.
\Leftrightarrow \nonumber \\
& q\neq 1 \ \text{and}\ \forall\  m,n\in\mathbb{Z}:\
\left\{
\begin{array}{ll}
& q^{m-n}=1  \\
& \left\{
\begin{array}{ll}
& q^{n+m-k}=-1; \\
& q^{m-k}=-1 \ \text{or} \ q^{n-k}=-1
\end{array}\right. \ \text{if} \ q^{m-n}\neq 1,\
\end{array}
\right.
\Leftrightarrow \nonumber \\
& q\neq 1 \ \text{and}\ \forall\  m,n\in\mathbb{Z}:\
\left\{
\begin{array}{ll}
& q^{m-n}=1;  \\
& \left\{
\begin{array}{ll}
& q^{n+m-k}=-1   \\
& q^{m-k}=-1,
\end{array}\right. \ \text{if} \ q^{m-n}\neq 1,\ q^{n-k}\neq -1
\\
& \left\{
\begin{array}{ll}
& q^{n+m-k}=-1; \\
& \ q^{n-k}=-1
\end{array}\right. \ \text{if} \ q^{m-n}\neq 1,\ q^{m-k}\neq -1,
\\
\end{array}
\right.
\Leftrightarrow \nonumber \\
\label{Vmultcondprfqcond}
& q\neq 1 \ \text{and}\ \forall\  m,n\in\mathbb{Z}:\
\left\{
\begin{array}{ll}
& q^{m-n}=1;  \\
& q^{n}= 1   \ \text{for} \  q^{m}\neq 1, q^{k}\neq -1;
\\
& q^{m}=1 \ \text{for} \  q^{n}\neq 1, q^{k}\neq -1,
\end{array}
\right.
\end{align}
If $q\neq 1$ and $q^{k}= -1$, then  \eqref{Vmultcondprfqcond} reduces to
$q\neq 1 \ \text{and}\ \forall\  m,n\in\mathbb{Z}: q^{m-n}=1$, which
does not hold because $q^{m-n}=-1\neq 1$ when $m-n= k$.
If $q\neq 1$, $q^{k}\neq -1$ and $q^{k}= 1$, the \eqref{Vmultcondprfqcond} does not hold
since for $m=2k+1$ and $n=k$,
$$q^{m-n}=q^{k+1}=q\neq 1, \
q^{m}=q^{2k+1}=q\neq 1, \
q^{n}=q^{k+1}=q\neq 1,
$$
and if $q\neq 1$, $q^{k}\neq -1$ and $q^{k}\neq 1$,
then \eqref{Vmultcondprfqcond} does not hold since for $m=2k$ and $n=k$,
$$q^{m-n}=q^{k}\neq 1, \
q^{m}=q^{2k}\neq 1, \
q^{n}=q^{k}=q\neq 1.
$$
Hence, if $k=\deg\alpha_k \neq 0$, then $(\mathcal {V}^{q}, [\cdot,\cdot], \alpha_k)$ is not multiplicative for any $q$.
\end{proof}

\begin{prop}
The $q$-deformed $W(2,2)$ Hom-Lie algebra $(\mathcal {W}^{q}, [\cdot,\cdot], \alpha)$ is multiplicative if and only if $q^2=-1$  \textup{(}which is equivalent to
$q=\pm i$ if there exists $i\in \mathbb{K}$ such that $ i^2=-1$, for example when
$\mathbb{K}$ is algebraically closed field, like $\mathbb{C}$\textup{)}.
\end{prop}
\begin{proof} For all $n,m\in\mathbb{Z}$, we have
\begin{align*}
    &[\alpha(L_m),\alpha(L_n)]-\alpha([L_m,L_n])=[(q^m+q^{-m})L_m,(q^n+q^{-n})L_n]-\alpha([m-n]L_{m+n})\\
     &\quad=(q^m+q^{-m})(q^n+q^{-n})[m-n]L_{m+n}-[m-n](q^{m+n}+q^{-m-n})L_{m+n}\\
      &\quad=[m-n](q^{m-n}+q^{n-m})L_{m+n}=
      \left\{\begin{array}{l}
      \frac{(q^{m-n}-q^{n-m})(q^{m-n}+q^{n-m})}{q-q^{-1}}L_{m+n}, \ \text{if}\ q\neq \pm 1\\
       2(m-n) L_{m+n}, \ \text{if}\ q=1\\
       2(n-m) L_{m+n}\ \text{if}\ q=-1
      \end{array}\right.
      \\
         &[\alpha(L_m),\alpha(W_n)]-\alpha([L_m,W_n])
         =[(q^m+q^{-m})L_m,(q^n+q^{-n})W_n]-\alpha([m-n]W_{m+n})\\
     &\quad=(q^m+q^{-m})(q^n+q^{-n})[m-n]W_{m+n}-[m-n](q^{m+n}+q^{-m-n})W_{m+n}\\
      &\quad=[m-n](q^{m-n}+q^{n-m})W_{m+n}=
      \left\{\begin{array}{l} \frac{(q^{m-n}-q^{n-m})(q^{m-n}+q^{n-m})}{q-q^{-1}}W_{m+n},
      \ \text{if}\ q\neq \pm 1\\
      2(m-n)   W_{m+n}, \ \text{if}\ q=1\\
      2(n-m)  W_{m+n}\ \text{if}\ q=-1
      \end{array}\right.,\\
        &\alpha([W_m,W_n])-[\alpha(W_m),\alpha(W_n)]=0.
\end{align*}
So, by Theorem \ref{thm:multnonmultcnds} \ref{thm:i:multnonmultcnds},
\begin{align*}
& \text{$(\mathcal {W}^{q}, [\cdot,\cdot], \alpha)$ is multiplicative}\  \Leftrightarrow \\
& \forall\ m,n\in\mathbb{Z}:\
\left\{\begin{array}{l} (q^{m-n}-q^{n-m})(q^{m-n}+q^{n-m})=0, \ \text{if}\ q\neq \pm 1\\
m-n=0, \ \text{if}\ q=\pm 1
\end{array}\right.
\Leftrightarrow \\
& q\neq \pm 1 \ \text{and} \ \forall\ m,n\in\mathbb{Z}:\ (q^{2(m-n)}-q^{2(n-m)})=0\ \Leftrightarrow \\
& q\neq \pm 1 \ \text{and} \ \forall\ m,n\in\mathbb{Z}:\ q^{4(m-n)}=1 \Leftrightarrow
q\neq \pm 1 \ \text{and} \ \forall\ p\in\mathbb{Z}:\ q^{4p}=1 \Leftrightarrow \\
& q\neq \pm 1 \ \text{and} \ \forall\ p\in\mathbb{Z}:\ q^4=1 \ \Leftrightarrow \ q^2=-1.\\
& \Leftrightarrow q=\pm i \ \text{if}\ \exists \ i\in \mathbb{K}: \ i^2=-1 \
\text{(for example if $\mathbb{K}$ is algebraically closed).}
\qedhere
\end{align*}
\end{proof}

\begin{prop}
The Hom-Lie algebra $(\mathcal {W}^{q}, [\cdot,\cdot], \beta)$
is not multiplicative for any
$q\in \mathbb{K}\setminus \{0\}$.
\end{prop}
\begin{proof}
By Theorem \ref{thm:multnonmultcnds} \ref{thm:iv:multnonmultcnds},
\begin{align*}
& (\mathcal {W}^{q}, [\cdot,\cdot], \beta) \ \text{is multiplicative} \ \Leftrightarrow  \
\forall \ m,n\in\mathbb{Z}: \ \left\{
\begin{array}{llll}
\beta([L_m,L_n])=[ \beta(L_m),\beta(L_n)]=0\\
\beta([L_m,W_n])=[\beta(L_m),\beta(W_n)]=0,
\end{array}
\right. \Leftrightarrow \\
&
\forall \ m,n\in\mathbb{Z}: \ \left\{
\begin{array}{l}
(q^{m-k}+q^{k-m})(q^{n-k}+q^{k-n})[m-n]L_{m+n+2k}=0\\ (q^{m+n-k}+q^{k-m-n})[m-n]L_{m+n+k}=0, \\(q^{m-k}+q^{k-m})(q^{n-k}+q^{k-n})[m-n]W_{m+n+2k}=0\\ (q^{m+n-k}+q^{k-m-n})[m-n]W_{m+n+k}=0,
\end{array}
\right.
\Leftrightarrow \\
\\
&
q\neq \pm 1 \  \text{and} \
\forall \ m,n\in\mathbb{Z}:
\left\{
\begin{array}{lll}
(q^{m-k}+q^{k-m})(q^{n-k}+q^{k-n})(q^{m-n}+q^{n-m})=0,\\ (q^{m+n-k}+q^{k-m-n})(q^{m-n}+q^{n-m})=0,
\end{array}
\right.
\end{align*}
which does not hold since for $m=k+1, n=k$ it reduces to the impossible
$$q\neq \pm 1 \  \text{and} \
\left\{
\begin{array}{lll}
2(q+q^{-1})^2=0, \\
(q^{k+1}+q^{-(k+1)})(q+q^{-1})=0,
\end{array}
\right.
$$
Hence $(\mathcal {W}^{q}, [\cdot,\cdot], \beta)$ is not multiplicative for any
$q\in \mathbb{K}\setminus \{0\}$.
\end{proof}

\subsection{Averaging operators on Hom-algebras}
\begin{defn} An averaging operator on a Hom-algebra $(A,[\cdot,\cdot],\alpha)$ over $\mathbb K$ is a linear operator $P: A\rightarrow A$,
satisfying for all $x, y\in A,$
\begin{align}
 \label{Palphacom}
& \alpha\circ P=P\circ\alpha,
& \text{(commutativity of $P$ with $\alpha$)}  \\
  \label{avopaxiom}
& [P(x),P(y)] = P([P(x),y])=P([x,P(y)]),
&\text{(averaging operator axiom)}
 \end{align}
\end{defn}
\begin{rmk}
In skewsymmetric Hom-algebras, and thus in the Hom-Lie algebras in particular, the skewsymmetry of multiplication  \eqref{skewsymmetry} implies that  \eqref{avopaxiom} is equivalent to
\begin{equation}
\label{averaging operator 2}
 [P(x),P(y)] = P([P(x),y]),\;\;\forall\ x, y\in A.
\end{equation}
\end{rmk}

\begin{prop}
\label{thm:PAsubalg}
If $P$ is an averaging operator on a Hom-algebra $\mathcal{A}=(A,[\cdot,\cdot],\alpha)$, then
\begin{enumerate}
\item \label{thm:PAsubalg:i}
$(P(A),[\cdot,\cdot],\alpha)$ is a Hom-subalgebra of the Hom-Lie algebra $(A,[\cdot,\cdot],\alpha)$;
\item
\label{thm:PAsubalg:ii}
$[P(A),ker(P)] \subseteq ker(P) $ and $[ker(P),P(A)]  \subseteq ker(P) $.
\item
\label{thm:PAsubalg:iii}
If $P$ is surjective, that is if $P(A)=A$, then $ker(P)$ is a two-sided Hom-ideal in the Hom-algebra $\mathcal{A}$, meaning that
$[A , ker(P)] \subseteq ker(P) $, $[ker(P), A]  \subseteq ker(P) $ and $\alpha(ker(P))\subseteq ker(P)$.
\end{enumerate}
\end{prop}

\begin{proof}
Since $P$ is a linear operator, $P(A)$ is a linear subspace of $A$.
\begin{enumerate}
\item
By \eqref{Palphacom}, $P$ and $\alpha$ commute, and hence
$\alpha(P(x))=P(\alpha(x))\in P(A).$
Since, for any
$x,y\in P(A)$, there exist
$x',y'\in A$ such that  $x = P(x'), \ y = P(y'),$
the averaging operator axiom \eqref{avopaxiom} yields $[x,y] = [P(x'),P(y')] = P([P(x'),y']) \in P(A)$.
\item Let  $x = P(x')$ for some $x' \in A$.
If $y' \in ker(P)$, then \eqref{avopaxiom} yields
     $P([x,y']) = P([P(x'),y']) = [P(x'),P(y')] = 0$,
and hence, $[P(A) , ker(P)]  \subseteq ker(P) $.
Let  $y = P(y')$ for some $y' \in A$.
If $x' \in ker(P)$, then \eqref{avopaxiom} yields
     $P([x',y]) = P([x',P(y')]) = [P(x'),P(y')] = 0$, and hence  $[ker(P), P(A)]  \subseteq ker(P) $.
\item The first two inclusions are a special case of \ref{thm:PAsubalg:ii}, and $\alpha(ker(P))\subseteq ker(P)$ follows from commutativity of $\alpha$ and $P$.
\qedhere
      \end{enumerate}
\end{proof}

The defining axioms of Hom-Leibniz algebras and Hom-Lie algebras are multilinear in their arguments and are inherited by Hom-subalgebras.
\begin{cor}
Let $P$ be an averaging operator on a Hom-algebra $(A,[\cdot,\cdot],\alpha)$. Then,
if $(A,[\cdot,\cdot],\alpha)$ is a skewsymmetric Hom-algebra, or a Hom-Leibniz algebra or a Hom-Lie algebra,  then  $(P(A),[\cdot,\cdot],\alpha)$ is
also a skewsymmetric Hom-algebra, or a Hom-Leibniz algebra or a Hom-Lie algebra respectively.
\end{cor}

\begin{prop}\label{prop:Hom-Lieto Hom-Leibniz}
If $\mathcal{A}=( A, [\cdot,\cdot], \alpha)$ is a Hom-Leibniz algebra and $P$ is an averaging operator on $\mathcal{A}$, then with
$\{\cdot,\cdot\}:A\times A\rightarrow A$ defined for all $x,y\in A$ by $\{x,y\}=[P(x),y],$
\begin{enumerate}
\item the triple   $\mathcal{A}'=( A, \{\cdot,\cdot\}, \alpha)$ is a Hom-Leibniz algebra;
\item If $\mathcal{A}=( A, [\cdot,\cdot], \alpha)$ is a Hom-Leibniz algebra, then $\mathcal{A}'=( A, \{\cdot,\cdot\}, \alpha)$ is a Hom-Lie algebra if and only if $[P(x),y]=-[P(y),x]$ for all $x,y\in A$;
\item If $\mathcal{A}=( A, [\cdot,\cdot], \alpha)$ is a Hom-Lie algebra, and the averaging linear operator $P$ is surjective, that is $P(A)=A$, then
$\mathcal{A}'=( A, \{\cdot,\cdot\}, \alpha)$ is a Hom-Lie algebra.
\end{enumerate}
\end{prop}
\begin{proof}
Let $x,y,z\in A$, Then \eqref{Leibnizidentity} in
$\mathcal{A}'$ is proved as follows:
\begin{align*}
&\{\alpha(x),\{y,z\}\}-\{\{x,y\},\alpha(z)\}-\{\alpha(y),\{x,z\}\}\\
&=[P(\alpha(x)),[P(y),z]]-[P([P(x),y],\alpha(z)]-[P(\alpha(y)),[P(x),P(z)]]\\
&\begin{array}[b]{l}
\text{\scriptsize ($\alpha$, $P$ commute)}\\
=[\alpha(P(x)),[P(y),z]]-[[P(x),P(y)],\alpha(z)]
-[\alpha(P(y)),[P(x),z]]
\end{array}
\begin{array}[b]{l}
\text{\scriptsize ($\mathcal{A}$ is Hom-Leibniz algebra)} \\
=0.
\end{array}
\qedhere
\end{align*}
\end{proof}

\begin{prop} \label{prop:avopslincombcompinv}
Let $\{P_j\}_{1\leq j \leq n}$ be a finite set of averaging operators on a Hom-Lie algebra $(A,[\cdot,\cdot],\alpha)$, and $\{\lambda_j\}_{1\leq j \leq n}\subseteq \mathbb{K}$. Then
\begin{enumerate}
\item \label{item:avopslincomb}
The operator $S=\sum_{j=1}^n\lambda_jP_j$ is an averaging operator on $A$ if
\begin{eqnarray*}
\sum_{\substack{ j,k=1 \\  i\neq j }}^n
\lambda_j\lambda_k P_j([P_k(x),y]) = \sum_{\substack{ j,k=1 \\  j\neq k }}^n\lambda_j\lambda_k [P_j(x),P_k(y)]
\end{eqnarray*}
\item \label{item:avopscomp} If $P_j \circ P_k = P_k \circ P_j$ for $1\leq k,j\leq n$, then
$T={\displaystyle \prod_{j=1}^{n}} P_j = P_1\circ \dots \circ P_n$ is an averaging operator.
\item \label{item:avopspoly} If $P_j \circ P_k = P_k \circ P_j$ for $1\leq k,j\leq n$, then for any polynomial $F \in \mathbb{K}[t_1,\dots,t_n]$ with zero constant term $F(0,\dots,0)=0$,
the operator $F(P_1,\dots,P_n)$ is an averaging operator.
\item \label{item:avopsinv} If an averaging operator $P$ is invertible, then $P^{-1}$ is an averaging operator.
\end{enumerate}
\end{prop}

\begin{proof}
\ref{item:avopslincomb}
The map $S$ is a linear operator on $A$ as a linear combination of the linear operators $\{P_j\}_{1\leq j \leq n}$, and $\alpha\circ S= S\circ\alpha$, since  $\alpha\circ P_j=P_j \circ\alpha,\ 1\leq j \leq n$. For all
$x, y \in A$.
  \begin{align*}
S([S(x),y]) =& (\sum_{j=1}^n\lambda_jP_j) ([(\sum_{k=1}^n\lambda_kP_k)(x),y]) =  (\sum_{j=1}^n\lambda_jP_j) ([\sum_{k=1}^n\lambda_kP_k(x),y])
\\
= &  \sum_{j,k=1}^n\lambda_j\lambda_k P_j ([P_k(x),y]) \\
= & \sum_{j=1}^n\lambda_j^2 P_j ([P_k(x),y]) +
\sum_{\substack{ j,k=1 \\  j\neq k}}^n\lambda_j\lambda_k P_j ([P_k(x),y]) \\
=& \sum_{j=1}^n\lambda_j^2 [P_j(x), P_j(y)]+
\sum_{\substack{ j,k=1 \\  j\neq k}}^n\lambda_j\lambda_k  ([P_j(x),P_k(y)]) \\
=& \sum_{j=1}^n [\lambda_j P_j(x), \lambda_j P_j(y)]
+
\sum_{\substack{ j,k=1 \\  j\neq k}}^n  ([\lambda_j P_j(x),\lambda_k P_k(y)]) \\
=& [\sum_{j=1}^n\lambda_jP_j,\sum_{k=1}^n\lambda_kP_k] =[S(x),S(y)].
\end{align*}
\newline
\ref{item:avopscomp} The operator $T_n =P_1 \circ \dots \circ P_{n}$ is linear as a composition of the linear operators, and also $\alpha \circ T_n = T_n\circ \alpha$ since $\alpha \circ P_j = P_j \circ \alpha$ for all $1\leq j \leq n$.
For $n=1$, $T_{1}={\displaystyle \prod_{j=1}^n} P_j = P_1$ is an averaging operator. Suppose that for $n = k$ the statement holds, that is $T_k=P_1 \circ \dots \circ  P_{k}$ is an averaging operator. Then, for all $x, y \in A$,
\begin{align*}
 T_{k+1}([T_{k+1}(x),y]) =& T_k \circ P_{k+1} ([T_k(P_{k+1}(x)),y]) \\^\text{($T_k,P_{k+1}$ commute)}=& T_k ( P_{k+1} ([P_{k+1}(T_k(x)),y]))\\  ^\text{($P_{k+1}$ is an averaging operator)}=& T_k ([ P_{k+1} (T_k(x)),P_{k+1}(y)])\\ ^\text{($T_k,P_{k+1}$ commute)}=& T_k ( [T_k (P_{k+1}(x)),P_{k+1}(y)])
\\  ^\text{($T_k$ is an averaging operator)}=&[T_k ( P_{k+1} (x)),T_k(P_{k+1}(y))]\\ =& [T_{k+1}(x), T_{k+1}(y)].
\end{align*}
proving that $T_n$ is averaging operator for $n=k+1$, which completes the proof by the principle of mathematical induction.    \\
\ref{item:avopspoly} Since $F(P_1,\dots,P_n)$ is a linear combination of compositions of averaging operators, it is also an averaging operator by \ref{item:avopscomp} and \ref{item:avopslincomb}.\\
\ref{item:avopsinv} It is clear that if $P$ is invertible, then $P^{-1}$ is a linear map of $A$ and $\alpha\circ P^{-1}=P^{-1}\circ\alpha$.
 Let $x, y \in A$.
Since $P$ is surjective, there exists $ y' \in A$ such that
 $P(x') = P^{-1}(x)$. Then,
 \begin{align*}
 P(P^{-1}([P^{-1}(x),y]))& = [P^{-1}(x),y]= [P^{-1}(x),P(P^{-1}(y)]\\
 & = [P(x'),P(P^{-1}(y)]= P([P(x'),P^{-1}(y)]) = P([P^{-1}(x),P^{-1}(y)]).
 \end{align*}
 Since $P$ is injective, we have
  $P^{-1}([P^{-1}(x),y]) = [P^{-1}(x),P^{-1}(y)]$.
\end{proof}

\begin{rmk}
By Proposition \ref{prop:avopslincombcompinv}, if $P$ and $Q$ are two averaging operators on a Hom-Lie algebra $(A,[\cdot,\cdot],\alpha)$, then
\begin{enumerate}
\item for any $\lambda\in \mathbb{K}$, $\lambda P$ is an averaging operator on $A$;
\item If $P([Q(x),y]) + Q([P(x),y]) = [Q(x),P(y)] + [P(x),Q(y)]$  holds
    for all $x, y \in A$, then $P + Q$ is an averaging operator.
\label{2avgOprsomme}
\item If $P \circ Q$ = $Q \circ P$, then $P \circ Q$ is an averaging operator;
\label{2avgOprcomposition}
\item for any polynomial $F \in \mathbb{K}[t]$ with zero constant term $F(0)=0$, the operator $F(P)$ is an averaging operator.
\end{enumerate}
\end{rmk}
\begin{rmk}
It is known that $P$ is an averaging operator on a Hom-Lie algebra $(A,[\cdot,\cdot],\alpha)$ if and only if $\lambda P$ is an averaging operator on $A$ for $\lambda\in\mathbb{K}^{\ast}=\mathbb{K}\setminus \{0\}$. So the set of averaging operators on any Hom-Lie algebra carries an action of $\mathbb{K}^{\ast}$ by scalar multiplication.
\end{rmk}

Next we provide the necessary and sufficient conditions for an idempotent linear operator to be an averaging operator.
\begin{defn} An idempotent operator on a Hom-Lie algebra $(A,[\cdot,\cdot],\alpha)$ over $\mathbb{K}$ is a linear map $P: A\rightarrow A$ satisfying $\alpha\circ P=P\circ\alpha$ and $P^{2}=P.$
\end{defn}

\begin{rmk}
Recall that there is a bijection
$$ \{\text{idempotent linear operators on } A\}\leftrightarrow
\{\text{direct sum decompositions}\ A=A_0\oplus A_1\}$$
where $A_0=im(P)$ and $A_1=ker(P)$.
The linear map $P$ corresponding to $A=A_0\oplus A_1$ is called the projection onto $A_0$ along $A_1$.
If $P$ is the projection onto $A_0$ along $A_1$, then
$I-P$ is the projection onto $A_1$ along $A_0$ since
$(I-P)^2=I-2P+P^2=I-2P+P=I-P$ and $im(I-P)=ker(P),\ Ker(I-P)=im(P)$.
\end{rmk}

\begin{prop}
Let $(A,[\cdot,\cdot],\alpha)$ be a Hom-Lie algebra and let $P:A\to A$ be an idempotent linear map. Let $A=A_0\oplus A_1$ be the corresponding linear decomposition. Then $P$ is an averaging operator if and only if \begin{equation}\label{eq:grad}
[A_0,A_0]\subseteq A_0, \quad [A_0,A_1]\subseteq A_1.
\end{equation}
\end{prop}
\begin{proof}
For any $x, y\in A$, denote $x=x_0+x_1$ and $y=y_0+y_1$ with $x_i, y_i\in A_i, i=0, 1$.
Suppose $P$ is an averaging operator. Then from $P(A)=A_0$ and $[P(x),P(y)]=P([P(x),y]),$ we obtain $[A_0,A_0]\subseteq A_0$.
Then we have
\begin{eqnarray*}
[P(x),P(y)]&=&[x_0,y_0], \\ P([P(x),y])&=&P([x_0,y_0]+[x_0,y_1])=P([x_0,y_0])+P([x_0,y_1])=[x_0,y_0]+P([x_0,y_1]).
\end{eqnarray*}
Thus from \eqref{averaging operator 2} we obtain  $P([x_0,y_1])=0$ for all $x_i,y_i\in A_i, i=0, 1$. Therefore \eqref{eq:grad} holds since $A_1=ker P$ by the definition of $P$.
Conversely, suppose \eqref{eq:grad} holds. Then we have
$$ P([P(x),y])=P([x_0,y_0]+[x_0,y_1])=P([x_0,y_0])+P([x_0,y_1])=[x_0,y_0]=[P(x),P(y)].$$
Thus $P$ is an averaging operator.
\end{proof}
\begin{cor}
An idempotent endomorphism $P:A\to A$ is an averaging operator.
\end{cor}
\begin{proof}
Let $A_0:=im P$ and $A_1:=ker P$. Then we have $A=A_0\oplus A_1$ and $P$ is the projection to $A_0$ along $A_1$. Since $A_1$ is an ideal of $A$, then \eqref{eq:grad} holds. Hence $P$ is an averaging operator.
\end{proof}

\section{On homogeneous averaging operators on \hmath $q$-deformed Witt Hom-algebra}\label{section3}
In this section we classify the  homogeneous averaging operators on the $q$-deformed
Witt algebra $\mathcal{V}^{q}$ and we give the induced Hom-Leibniz algebras
from the averaging operators on the $q$-deformed
Witt algebra $\mathcal{V}^{q}$ and its multiplicativity condition is studied.
    \begin{defn}
   A homogeneous operator $F$ with degree $d\in\mathbb{Z}$ on the
$q$-deformed Witt Hom-algebra $\mathcal{V}^{q}$ is a linear operator on $\mathcal{V}^{q}$ satisfying
$F(\mathcal{V}^{q}_m)\subseteq \mathcal{V}^{q}_{m+d}$ for all $m\in\mathbb{Z}$.
    \end{defn}
Therefore, a homogeneous averaging operator $P_d$ with degree $d$ on the $q$-deformed Witt Hom-algebra
$\mathcal{V}^{q}$ is an averaging operator on $\mathcal{V}^{q}$ of the following form
\begin{equation}\label{averaging Witt}
    P_{d}(L_{m})=f(m+d)L_{m+d},\quad \forall\   m\in\mathbb{Z},
\end{equation}
where $f$ is a $\mathbb{K}$-valued function defined on $\mathbb{Z}.$

Let $P_d$ be a homogeneous averaging operator with degree $d$ on the $q$-deformed Witt Hom-algebra
$(\mathcal{V}^{q},[\cdot,\cdot],\alpha)$ satisfying \eqref{averaging Witt}. Then by \eqref{witt bracket} and \eqref{averaging operator 2},
  \begin{align*} {} %
 [P_{d}(L_m),P_{d}(L_n)]
  &  =[f(m+d)L_{m+d},f(n+d)L_{n+d}] \\
& =f(m+d)f(n+d)(\{m+d\}-\{n+d\})L_{m+n+2d}, \\
 P_d([P_d(L_m),L_n])
  & =P_d([f(m+d)L_{m+d},L_{n}]) \\
&    =f(m+d)f(m+n+2d)(\{m+d\}-\{n\})L_{m+n+2d}.
\end{align*}
We see that the function $f$ satisfies for all $m,n\in \mathbb{Z}$,
$$ f(m+d)f(n+d)(\{m+d\}-\{n+d\}) = f(m+d)f(m+n+2d)(\{m+d\}-\{n\}),$$
or equivalently, after changing $m\rightarrow m-d$ and $n\rightarrow n-d$,
\begin{equation}
 f(m)\big(f(n)  (\{m\}-\{n\})-f(m+n)(\{m\}-\{n-d\})\big)=0. \label{firsteqwittgeneral}
\end{equation}

\begin{lem}\label{commuteWitt}
If $P_{d}$ is a non-zero averaging operator on the $q$-deformed Witt Hom-algebra
$(\mathcal{V}^{q},[\cdot,\cdot],\alpha)$ with degree $d$, then
$\alpha\circ P_{d}=P_{d}\circ\alpha$ if and only if
$q^d = 1$.
\end{lem}
\begin{proof}
For all $m\in\mathbb{Z}$,
\begin{align*}
& \alpha\circ P_d(L_m)=\alpha(f(m+d)L_{m+d})
=f(m+d)\alpha(L_{m+d})
=f(m+d)(1+q^{m+d})L_{m+d},\\
& P_d\circ \alpha(L_m)=P_d( \alpha(L_m))=P_d((1+q^{m})L_m)
=(1+q^{m})P_d(L_m)
=(1+q^{m})f(m+d)L_{m+d},\\
& \alpha\circ P_d(L_m)=P_d\circ \alpha(L_m)\
\Leftrightarrow
\forall\ m\in\mathbb{Z}:\ f(m+d)(q^{d}-1)=0,
\end{align*}
is equivalent to $q^{d} = 1$ when $P_d\neq 0$, as in this case
$f(m+d)\neq 0$ for some $m\in \mathbb{Z}$.
\end{proof}
\subsection*{Case 1: $q=1$}

When $q=1$, equation \eqref{firsteqwittgeneral} becomes for all $m,n\in \mathbb{Z}$,
\begin{equation}
 f(m)\big(f(n)  (m-n)-f(m+n)(m-n+d)\big)=0. \label{firsteqwitt1}
\end{equation}
Plugging $n = 0$ in \eqref{firsteqwitt1}, we have
\begin{equation}
 f(m)\big(mf(0) -(m+d)f(m)\big)=0. \label{firsteqwitt2}
\end{equation}
\subsubsection*{Subcase 1: $d=0$}
\begin{prop}\label{propwitt1}
With the notations as above, the averaging operator $P_0$
with degree $d=0$ is given by
\begin{equation*}
f(m)=\mu f(0)+\nu\delta_{m,0},\;\; \forall\ m\in\mathbb{Z},~for~ some~(\nu,\mu)\in \mathbb{K}\times\{0,1\},
\end{equation*}
where for any $x,y\in\mathbb{K},$ $\delta_{x,y}=\left\{
  \begin{array}{lllll}
1 \text{ if } x=y\\
0 \text{ if } x\neq y.
  \end{array}\right.$
\end{prop}
\begin{proof}
When $d=0$, \eqref{firsteqwitt2} becomes $m f(m)(f(0)-f(m))=0$.
Hence
$f(m)=\mu f(0)+\nu\delta_{m,0}$ for some $\nu\in \mathbb{K}$ and $\mu\in\{0,1\}$.
\end{proof}
\subsubsection*{Subcase 2: $d\in\mathbb{Z}^\ast$}
\begin{prop}\label{propwitt2}
With the notations as above, when the degree $d\in\mathbb{Z}^\ast$ and $f(0)=0$, we have
\begin{equation*}
f(m)=\nu\delta_{m+d,0},\;\; \forall\ m\in\mathbb{Z},~where~\nu\in \mathbb{K},
\end{equation*}
\end{prop}
\begin{proof}
When $d\neq 0$ and $f(0)=0$, then by equation \eqref{firsteqwitt2}, we have for all $m\in\mathbb{Z}$, $(m+d)f^{2}(m)=0$. Thus the function $f$ satisfies for any $m\in\mathbb{Z}$,
$f(m)=\nu\delta_{m+d,0}$ for some $\nu\in \mathbb{K}$.
\end{proof}
\begin{prop}\label{propwitt3}
With the notations as above, when the degree $d\in\mathbb{Z}^\ast$ and $f(0)\neq0$, we have
\begin{equation*}
f(m)=\mu\frac{m}{m+d}f(0)\delta_{m,\mathbb{Z}\setminus\{-d\}},\;\; \forall\ m\in\mathbb{Z},~where~\mu\in\{0,1\},
\end{equation*}
where  $\delta_{m,\mathbb{Z}\setminus\{-d\}}=\left\{
  \begin{array}{lllll}
1 \text{ if } m\neq -d\\
0 \text{ if } m=-d.
  \end{array}\right.$
\end{prop}
\begin{proof}
When $d\neq 0$ and $f(0)\neq 0$, it follows from equation \eqref{firsteqwitt2} for $m=-d$ that $f(-d)=0$. Moreover for $m\neq -d$ in equation \eqref{firsteqwitt2}, we have
$
f(m)=\mu\frac{m}{m+d}f(0),\;\; \text{ for some }\mu\in\{0,1\}.
$
Therefore for all $m\in\mathbb{Z}$,
$
f(m)=\mu\frac{m}{m+d}f(0)\delta_{m,\mathbb{Z}\setminus\{-d\}},\;\;  where ~\mu\in\{0,1\}.
$
\end{proof}
\subsection*{Case 2: $q\neq 1$ and $q^d=1$}
\begin{prop}\label{propWitt}
With the notations as above, Suppose that the degree $d$ satisfying $q\neq 1$ and $q^{d}=1$, then we have
\begin{equation*}
f(m)=\mu f(0)+\nu\delta_{q^{m},1},\;\; \forall\ m\in\mathbb{Z},~where~(\nu,\mu)\in \mathbb{K}\times\{0,1\}.
\end{equation*}
\end{prop}
\begin{proof}
When $q\neq 1$ and $q^{d}=1$, \eqref{firsteqwittgeneral} becomes for all $m,n\in \mathbb{Z},$
\begin{equation}
 f(m)\big(f(n)(\{m\}-\{n\})-f(m+n)(\{m\}-\{n\})\big)=0\label{first eq witt}.
\end{equation}
Taking $n=0$ in \eqref{first eq witt} yields
$\{m\}f(m)(f(0)-f(m))=0.$
Hence $f(m)=\mu f(0)+\nu\delta_{q^{m},1}$ for some $\nu\in \mathbb{K}$ and $\mu\in\{0,1\}$.
\end{proof}

\begin{thm}\label{Thm2.14}
The homogeneous averaging operator $P_d$ with degree $d$ on the $q$-deformed Witt Hom-algebra $(\mathcal{V}^{q},[\cdot,\cdot],\alpha)$ must be one of the following operators, given for all $m\in \mathbb{Z}$, by
\begin{align*}&P_d^{1}(L_m) =\beta+\nu\delta_{m+2d,0}L_{m+d},\quad\quad \text{ for } q=1 \text{ and }d\in\mathbb{Z},\\
&P_d^{2}(L_m) =
\mu\frac{m+d}{m+2d}\gamma \delta_{m+d,\mathbb{Z}\setminus\{-d\}}L_{m+d},\quad\quad \text{ for } q=1 \text{ and }d\in\mathbb{Z},\\
&P_d^{3}(L_m) =
\beta+\nu\delta_{q^{m},1}L_{m+d},\quad\quad \text{ for } q\neq 1 \text{ and } q^d=1, \end{align*}
where $\beta,\nu\in\mathbb{K}$, $\gamma\in\mathbb{K}^{\ast}$ and $\mu\in\{0,1\}$.
\end{thm}
\begin{proof}
Directly by ombining Lemma \ref{commuteWitt} and Propositions \ref{propwitt1}, \ref{propwitt2}, \ref{propwitt3} and \ref{propWitt}.
\end{proof}
Now, using the construction given in Proposition \ref{prop:Hom-Lieto Hom-Leibniz}, we have the following:
\begin{thm}
The homogeneous averaging operators obtained in Theorem \ref{Thm2.14}
for the $q$-deformed Witt Hom-algebra $(\mathcal{V}^{q},[\cdot,\cdot],\alpha)$, give rise to the following Hom-Leibniz algebras on the underlying
linear space $\mathcal{V}^q$,
\begin{align*}
&\{L_m,L_n\}^{1}=( \beta+\nu\delta_{m+d,0})(m-n)L_{m+n},\quad\quad \forall m,n\in\mathbb{Z}, \text{ where } q=1 \text{ and }d\in\mathbb{Z},\\
&\{L_m,L_n\}^{2}=(\mu\frac{m}{m+d}\gamma\delta_{m,\mathbb{Z}\setminus\{-d\}})(m-n)L_{m+n},\quad\quad \forall m,n\in\mathbb{Z}, \text{ where } q=1 \text{ and }d\in\mathbb{Z},\\
   & \{L_m,L_n\}^{3}=( \beta+\nu\delta_{q^{m},1})(\{m\}-\{n\})L_{m+n},\quad\quad\forall m,n\in\mathbb{Z}, \text{ where } q\neq 1 \text{ and } q^d=1,
\end{align*}
where $\beta,\nu\in\mathbb{K}$, $\gamma\in\mathbb{K}^{\ast}$ and $\mu\in\{0,1\}$.
\end{thm}
\begin{proof}
By Proposition \ref{prop:Hom-Lieto Hom-Leibniz}, the  Hom-Leibniz algebra induced by $P_d^{i},~i=1,2,3,$ is given for $m,n\in\mathbb{Z}$ by
\begin{align*}
&\{L_m,L_n\}^{1}=[P_d^{1}(L_m),L_n]
=[(\beta+\nu\delta_{m+2d,0}L_{m+d},L_n]=( \beta+\nu\delta_{m+2d,0})(m+d-n)L_{m+n+d}\\&=( \beta+\nu\delta_{m+d,0})(m-n)L_{m+n},\\[0.2cm]
&\{L_m,L_n\}^{2}=[P_d^{2}(L_m),L_n]=[\mu\frac{m+d}{m+2d}\gamma\delta_{m+d,\mathbb{Z}\setminus\{-d\}}L_{m+d},L_n]\\
&=(\mu\frac{m+d}{m+2d}\gamma\delta_{m+d,\mathbb{Z}\setminus\{-d\}})(m+d-n)L_{m+n+d}
=(\mu\frac{m}{m+d}\gamma\delta_{m,\mathbb{Z}\setminus\{-d\}})(m-n)L_{m+n},\\[0.2cm]
&  \{L_m,L_n\}^{3}=[P_d^{3}(L_m),L_n]=[(\beta+\nu\delta_{q^{m},1})L_{m+d},L_n]=(\beta+\nu\delta_{q^{m},1})(\{m\}-\{n\})L_{m+n+d}\\
  &=(\beta+\nu\delta_{q^{m},1})(\{m\}-\{n-d\})L_{m+n}
  \begin{array}[b]{l}
  \text{\scriptsize ($\{n-d\}=\{n\}+q^{n}\{-d\}$)}\\ =(\beta+\nu\delta_{q^{m},1})(\{m\}-(\{n\}+q^{n}\{-d\})L_{m+n}
  \end{array} \\
&
  ^\text{($\{-d\}=0$)}=(\beta+\nu\delta_{q^{m},1})(\{m\}-\{n\})L_{m+n}. \qedhere
\end{align*}
\end{proof}
\begin{prop}
The non-trivial Hom-Leibniz algebra $(\mathcal {V}^{q}, ~\{\cdot,\cdot\}^i, \alpha)$ induced by $P_d^{i}$ is multiplicative if and only if $i=3$ and $q=-1$.
\end{prop}
\begin{proof}
By Corollary \ref{cor:Hom-algebramultiplicativity} item \ref{item2:corollary}:
\begin{enumerate}[label=,leftmargin=0pt]
\item $(\mathcal {V}^{q}, \{\cdot,\cdot\}^1, \alpha)
\ \text{is multiplicative}\ \Leftrightarrow
\forall\ m,n\in\mathbb{Z}:\ 2(\beta+\nu\delta_{m+d,0})(m-n)=0\ \Leftrightarrow $\\
$ \forall\ m\in\mathbb{Z}:\ (\beta+\nu\delta_{m+d,0})=0 \Leftrightarrow
\beta=\nu=0\Leftrightarrow \{\cdot,\cdot\}^1=0.
$
\item
$\text{$(\mathcal {V}^{q}, \{\cdot,\cdot\}^2, \alpha)$ is multiplicative} \Leftrightarrow
\forall\ m,n\in\mathbb{Z}:\ 2\mu\frac{m}{m+d}\gamma\delta_{m,\mathbb{Z}\setminus\{-d\}})(m-n)=0\ \Leftrightarrow $
$ \forall\ m\in\mathbb{Z}:\ \mu\frac{m}{m+d}\gamma\delta_{m,\mathbb{Z}\setminus\{-d\}})=0
 \stackrel{\gamma\neq0}{\Leftrightarrow}
\mu=0
\Leftrightarrow \{\cdot,\cdot\}^2=0.
$
\item
$ \text{$(\mathcal {V}^{q}, \{\cdot,\cdot\}^3, \alpha)$ is multiplicative}\Leftrightarrow $\\
$\forall\ m,n\in\mathbb{Z}:\ (\beta+\nu\delta_{q^{m},1})(\{m\}-\{n\})\big((1+q^{m+n})-(1+q^{m})(1+q^{n})\big)=0\ \Leftrightarrow $ \\
$ \forall\ m\in\mathbb{Z}:\ (\beta+\nu\delta_{q^{m},1})(q^{n}-q^{m})(q^{n}+q^{m})=0 \Leftrightarrow $\\
$ \forall\ m\in\mathbb{Z}:\ (\beta+\nu\delta_{q^{m},1})(q^{2n}-q^{2m})=0\Leftrightarrow $\\
$ \forall\ m,n\in\mathbb{Z}:\ \left\{
  \begin{array}{ll}
    (q^{2(n-m)})=1,\text{ or } \\
    (\beta+\nu\delta_{q^{m},1})=0,
  \end{array}
\right.\Leftrightarrow $\\
$\left\{
  \begin{array}{ll}
   q^{2p}=1,~\forall p\in\mathbb{Z} ,\text{ or } \\
   \forall\ m\in\mathbb{Z}:\ \beta+\nu\delta_{q^{m},1}=0,
  \end{array}
\right.
\Leftrightarrow
\left\{
  \begin{array}{ll}
   q=-1 ,\text{ or }\hspace{1cm}(\text{since }q\neq1) \\
   \beta=\nu=0.
  \end{array}
\right. $ \qedhere
\end{enumerate}
\end{proof}

\section{On homogeneous averaging operators on \hmath $q$-deformed
\hmath $W$(2,2) Hom-algebra}\label{section4}
In this section we classify the homogeneous averaging operators on the $q$-deformed $W(2,2)$ Hom-algebra $\mathcal{W}^{q}$. Also, we give the induced Hom-Leibniz algebras
from the averaging operators on the $q$-deformed $W(2,2)$ Hom-algebra and its multiplicativity condition is studied.
\begin{defn}
A homogeneous operator $F$ with degree $d\in\mathbb{Z}$ on the
$q$-deformed $W(2,2)$ Hom-algebra $\mathcal{W}^{q}$ is a linear operator on $\mathcal{W}^{q}$ satisfying
$F(\mathcal{W}^{q}_{m}) \subseteq \mathcal{W}^{q}_{m+d}$
for all $m \in \mathbb{Z}$.
\end{defn}
Hence a homogeneous averaging operator $P_d$ with degree $d$ on the $q$-deformed $W(2,2)$ Hom-algebra $\mathcal{W}^{q}$
is an averaging operator on $\mathcal{W}^{q}$ with the following form:
\begin{eqnarray}
    P_{d}(L_{m})=f_1(m+d)L_{m+d}+g_1(m+d)W_{m+d}\label{W operators 1},\\
     P_{d}(W_{m})=f_2(m+d)L_{m+d}+g_2(m+d)W_{m+d}\label{W operators 2},
\end{eqnarray}
where $f_i$ and $g_i$ are $\mathbb{K}$-valued functions defined on $\mathbb{Z}$.

Let $P_d$ be a homogeneous averaging operator of degree $d$ on the $q$-deformed $W(2,2)$ Hom-algebra
$\mathcal{W}^{q}$ satisfying equations \eqref{W operators 1} and \eqref{W operators 2}. Then, by equations \eqref{H-L1} and
\eqref{averaging operator 2},
\begin{align*}
   & [P_d(L_m),P_d(L_n)]=[f_1(m+d)L_{m+d}+g_1(m+d)W_{m+d},f_1(n+d)L_{n+d}+g_1(n+d)W_{n+d}]\\
    &\quad =f_1(m+d)f_1(n+d)[m-n]L_{m+n+2d}+f_1(m+d)g_1(n+d)[m-n]W_{n+m+2d}\\
    &\hspace{8cm} -g_1(m+d)f_1(n+d)[n-m]W_{m+n+2d},\\*[0,2cm]
   & P_d([P_d(L_m),L_n])=P_d([f_1(m+d)L_{m+d}+g_1(m+d)W_{m+d},L_n])\\
    &\quad=f_1(m+d)[m+d-n]P_d(L_{m+n+d})-g_1(m+d)[n-m-d]P_d(W_{m+n+d})\\
    &\quad=f_1(m+d)f_1(m+n+2d)[m+d-n]L_{m+n+2d}\\&\quad\quad+f_1(m+d)g_1(m+n+2d)[m+d-n]W_{m+n+2d}\\
    &\quad\quad-g_1(m+d)f_2(m+n+2d)[n-m-d]L_{m+n+2d}\\
    &\quad\quad-g_1(m+d)g_2(m+n+2d)[n-m-d]W_{m+n+2d},
\end{align*}
we see that the functions $f_i$ and $g_i$ satisfy, for all $m,n\in \mathbb{Z}$, the following equations:
\begin{align}
&f_1(m)f_1(n)[m-n]-f_1(m)f_1(m+n)[m+d-n]+g_1(m)f_2(m+n)[n-m-d]=0,\label{first eq}\\
&
\begin{array}{l}
f_1(m)g_1(n)[m-n]-f_1(n)g_1(m)[n-m]-f_1(m)g_1(m+n)[m+d-n] \\
\hspace{8cm} +g_1(m)g_2(m+n)[n-m-d]=0.
\end{array}
\label{second eq}
\end{align}
and from
\begin{align*}
  & [P_d(L_{m}),P_d(W_n)]=[f_1(m+d)L_{m+d}+g_1(m+d)W_{m+d},f_2(n+d)L_{n+d}+g_2(n+d)W_{n+d}]\\
   &\quad=f_1(m+d)f_2(n+d)[m-n]L_{m+n+2d}+f_1(m+d)g_2(n+d)[m-n]W_{m+n+2d}\\
   &\hspace{8cm}-g_1(m+d)f_2(n+d)[n-m]W_{m+n+2d},\\*[0,2cm]
   &P_d([P_d(L_m),W_n])=P_d([f_1(m+d)L_{m+d}+g_1(m+d)W_{m+d},W_n])\\
   &\quad=f_1(m+d)f_2(m+n+2d)[m+d-n]L_{m+n+2d}\\
   &\hspace{6cm}+f_1(m+d)g_2(m+n+2d)[m+d-n]W_{m+n+2d},
   \end{align*}
we see that the functions $f_i$ and $g_i$ satisfy, for all $m,n\in \mathbb{Z}$, the following equations:
\begin{align}
&f_{1}(m)f_2(n)[m-n]-f_1(m)f_2(m+n)[m+d-n]=0,\label{fifth eq}\\
&f_1(m)g_2(n)[m-n]-g_1(m)f_2(n)[n-m]-f_1(m)g_2(m+n)[m+d-n]=0.\label{sixth eq}
\end{align}
In the same way, from
\begin{align*}
   & [P_d(W_m),P_d
(W_n)]=[f_2(m+d)L_{m+d}+g_2(m+d)W_{m+d},f_2(n+d)L_{n+d}+g_2(n+d)W_{n+d}]\\
   &\quad =[f_2(m+d)f_2(n+d)[m-n]L_{n+m+2d}+f_2(m+d)g_2(n+d)[m-n]W_{m+n+2d}\\
    &\hspace{8cm} -g_2(m+d)f_2(n+d)[n-m]W_{m+n+2d},\\*[0,2cm]
    &P_d
([P_d
(W_{m}),W_n])=P_d
([f_2(m+d)L_{m+d}+g_2(m+d)W_{m+d},W_n]\\
   &\quad=f_2(m+d)f_2(m+n+2d)[m+d-n]L_{m+n+2d}+f_2(m+d)g_2(m+n+2d)W_{m+n+2d},\\*[0,2cm]
&P_d
([W_m,P_d
(W_m])=P_d
([W_m,f_2(n+d)L_{n+d}+g_2(n+d)W_{n+d}])\\
&\quad\quad-f_2(n+d)f_2(m+n+2d)[m-n-d]L_{m+n+2d}\\
&\quad\quad-f_2(n+d)g_2(m+n+2d)[m-n-d]W_{m+n+2d},
\end{align*}
we see that the functions $f_i$ and $g_i$ satisfy for all $m,n\in \mathbb{Z}$, the following equations:
\begin{align}
    &f_2(m)f_2(n)[m-n]-f_2(m)f_2(m+n)[m+d-n]=0,\label{nineth eq}\\
    &f_2(m)g_2(n)[m-n]-f_2(n)g_2(m)[n-m]-f_2(m)g_2(m+n)[m+d-n]=0.\label{tenth eq}
\end{align}
Similarly from
\begin{align*}
   & [P_d
(W_m),P_d
(L_n)]=[f_2(m+d)L_{m+d}+g_2(m+d)W_{m+d},f_1(n+d)L_{n+d}+g_1(n+d)W_{n+d}]\\
    &\quad =f_2(m+d)f_1(n+d)[m-n]L_{m+n+2d}\\
    &\quad\quad+f_2(m+d)g_1(n+d)[m-n] W_{m+n+2d}-g_2(m+d)f_1(n+d)[m-n]W_{m+n+2d},\\*[0,2cm]
&P_d
([P_d
(W_{m}),L_{n}])=P_d
([f_2(m+d)L_{m+d}+g_2(m+d)W_{m+d},L_n])\\
&\quad=f_2(m+d)[m+d-n](f_1(m+n+2d)L_{m+n+2d}+g_1(m+n+2d)L_{m+n+2d})\\
&\quad\quad g_2(m+d)[n-d-m](f_2(m+n+2d)L_{m+n+2d}+g_2(m+n+2d)W_{m+n+2d}),
\end{align*}
we see that the functions $f_i$ and $g_i$ satisfy, for all $m,n\in \mathbb{Z}$,
\begin{align}
&f_2(m)f_1(n)[m-n]-f_2(m)f_1(m+n)[m+d-n]+g_2(m)f_2(m+n)[n-m-d]=0,\label{thirteenth eq}\\
&\begin{array}{l}
f_2(m)g_1(n)[m-n]-g_2(m)f_1(n)[n-m]-f_2(m)g_1(m+n)[m+d-n]\\
\hspace{8cm} +g_2(m)g_2(m+n)[n-m-d]=0.
\end{array}
\label{fourteenth eq}
\end{align}
\begin{lem}\label{lemma3}
If $P_{d}$ be a non zero averaging operator on $q$-deformed $W(2,2)$ Hom-algebra $\mathcal{W}^{q}$ with degree $d$. Then  $\alpha\circ P_{d}=P_{d}\circ\alpha$ if and only if $q^d=1$.
\end{lem}
\begin{proof}
For all $m\in\mathbb{Z},$ we have
\begin{align*}
\alpha\circ P_{d}(L_{m})&=\alpha(  P_{d}(L_{m}))=\alpha(f_1(m+d)L_{m+d}+g_1(m+d)W_{m+d})\\
&=f_1(m+d)(q^{m+d}+q^{-(m+d)})L_{m+d}+g_1(m+d)(q^{m+d}+q^{-(m+d)})W_{m+d},\\
P_{d}\circ\alpha(L_{m})&=P_{d}(\alpha(L_{m}))=P_{d}((q^{m}+q^{-m})L_{m})\\
&=f_1(m+d)(q^{m}+q^{-m})L_{m}+g_1(m+d)(q^{m}+q^{-m})W_{m}.
\end{align*}
Then $\forall m\in\mathbb{Z},~\alpha\circ P_{d}(L_m)=P_{d}\circ\alpha(L_m)$ if and only if $q^{m}+q^{-m}=q^{m+d}+q^{-m-d}.$  Similarly, $\forall m\in\mathbb{Z},~\alpha\circ P_{d}(W_m)=P_{d}\circ\alpha(W_m)$ if and only if $q^{m}+q^{-m}=q^{m+d}+q^{-m-d}.$ Thus,
\begin{enumerate}
\item
if $q^{d}=1$, it is clear that $\alpha\circ P_{d}=P_{d}\circ\alpha$,
\item
if $q^{d}\neq 1$, we have
\begin{align*}
&\alpha\circ P_{d}=P_{d}\circ\alpha\\
&\Longleftrightarrow
\forall m\in\mathbb{Z},~q^{m}+q^{-m}=q^{m+d}+q^{-m-d}
\Longleftrightarrow \forall m\in\mathbb{Z},~q^m(1-q^d)=q^{-m-d}(1-q^d)\\
&\begin{array}[b]{l}
\text{\scriptsize ($q^d\neq 1$)} \\
\Longleftrightarrow \forall m\in\mathbb{Z},~q^m=q^{-m-d}
\end{array}
\Longleftrightarrow  \forall m\in\mathbb{Z},~q^{2m+d}=1,
\end{align*}
this implies for $m=0$, $q^d=1$ which impossible since $q^d\neq 1$.
\qedhere
\end{enumerate}
\end{proof}
\subsection*{Case 1: $q\neq -1,1$ and $q^{d}=1$}
\subsubsection*{Subcase 1: $q^{2m}=1$}
  Take $n=0$ in \eqref{first eq}-\eqref{fourteenth eq}. For $q^{d}=1$ and $q^{2m}=1$, the functions $f_1,~f_2,~g_1$ and $g_2$ satisfy
$$
f_1(m)=\nu_1,\ f_2(m)=\nu_2,\ g_1(m)=\nu_3,\ g_2(m)=\nu_4,\ \nu_i\in\mathbb{K}.
$$
Then we have the following Proposition.
\begin{prop}\label{prop1:q neq 1,-1}
If $P_d$ is an averaging operator on $\mathcal{W}^{q}$ with degree $d$, where $q^{d}=1, q^{2m}=1$, then
$$
\left\{
  \begin{array}{ll}
  f_1(m)=\nu_1, &
    f_2(m)=\nu_2,\\
     g_1(m)=\nu_3,&
     g_2(m)=\nu_4,
  \end{array}\right.
\quad \text{where $\nu_i\in\mathbb{K}.$} $$
  \end{prop}
\subsubsection*{Subcase 2: $q^{2m}\neq 1$.}
Taking $n=0$ in \eqref{nineth eq} yields
$$
\begin{array}[t]{l} f_2(m)f_2(0)[m]=f_{2}^{2}(m)[m+d] =f_2^{2}(m)(q^{-d}[m]+q^{m}[d])
\begin{array}[t]{l}
=f_2^{2}(m)[m].\\
\text{\small (since $q^{d}=1$)}
\end{array}
\end{array}
$$
This gives $f_2(m)(f_2(0)-f_2(m))=0$. Hence,
\begin{equation*}
f_2(m)=\mu_1 f_2(0),\quad\mu_1\in\{0,1\}.
\end{equation*}
Then, we have the following Proposition.
\begin{prop}\label{prop2:q neq 1,-1}
If $P_d$ is an averaging operators on $\mathcal{W}^{q}$ with degree $d$ such that $q^{d}=1$, $q^{2m}\neq 1$ and $f_2(0)=0,$
then
\begin{enumerate}
    \item if $f_1(0)=0$, then
$f_1(m)=0,\
  f_2(m)=0,\
     g_1(m)=\gamma,\
  g_2(m)= 0,
$
where $\gamma\in\mathbb{K}$;
  \item  if $f_1(0)\neq 0$, then
\begin{enumerate}
  \item
 $ f_1(m)=f_1(0),\
  f_2(m)=0,\
     g_1(m)=\gamma,\
  g_2(m)= 0$, where $\gamma\in\mathbb{K};$
  \item
  $
  f_1(m)=f_1(0),\
  f_2(m)=0,\
     g_1(m)=0,\
  g_2(m)=f_1(0)$, where $\gamma\in\mathbb{K};$
  \item
  $
  f_1(m)=0,\
  f_2(m)=0,\
     g_1(m)=\gamma,\
  g_2(m)= f_1(0)$, where $\gamma\in\mathbb{K}.$
\end{enumerate}
  \end{enumerate}
\end{prop}

\begin{proof}
Let $q^{d}=1$, $q^{2m}\neq 1$ and $f_2(0)=0$ as assumed.
Taking $n=0$ in \eqref{first eq} yields
\begin{align*}
f_1(m)f_1(0)[m]=&f_1^{2}(m)[m+d]-g_1(m)f_2(m)[-m-d]\\
=&f_1^{2}(m)[m]+g_1(m)f_2(m)[m] . \quad \text{(since $q^{d}=1$)}
\end{align*}
Since $q^{2m}\neq 0$, we get $f_1(m)f_1(0)=f_1^{2}(m)+g_1(m)f_2(m)$, and since $f_2(m)=0$,
we obtain $f_1(m)(f_1(0)-f_1(m))=0$. Thus,
$f_1(m)=\mu_2 f_1(0)$, $\mu_2\in\{0,1\}$. Setting $n=0$ in \eqref{thirteenth eq} yields
\begin{align*}
f_2(m)g_1(0)[m]=&g_2(m)f_1(0)[-m]+f_2(m)g_1(m)[m+d]-g_2^{2}(m)[-m-d]\\
=&-g_2(m)f_1(0)[m]+f_2(m)g_1(m)[m]+g_2^{2}(m)[m]. \quad \text{(since $q^{d}=1$)}
\end{align*}
Since $q^{2m}\neq1$, we have $f_2(m)g_1(0)=-g_2(m)f_1(0)+f_2(m)g_1(m)+g_2^{2}(m)$, from which together with $f_2(m)=0$, we get
$g_2(m)=\mu_3 f_1(0)$, $\kappa_3\in\{0,1\}.$
Taking $n=0$ in \eqref{second eq} gives
\begin{align*}
f_1(m)g_1(0)[m]=&f_1(0)g_1(m)[-m]+f_1(m)g_1(m)[m+d]-g_1(m)g_2(m)[-m-d]\\
=&-f_1(0)g_1(m)[m]+f_1(m)g_1(m)[m]+g_1(m)g_2(m)[m], \quad \text{(since $q^{d}=1$)}
\end{align*}
from which with $q^{2m}\neq1$ we get $f_1(m)g_1(0)=-f_1(0)g_1(m)+f_1(m)g_1(m)+g_1(m)g_2(m)$, and with $f_1(m)=\mu_2 f_1(0)$ and $g_2(m)=\mu_3 f_1(0)$, we obtain $g_1(m)(\mu_2+\mu_3-1)f_1(0)=\mu_2 f_1(0)g_1(0)$. Then, we have the two cases:
\begin{enumerate}
    \item
    if $f_1(0)=0$, then
    $g_1(m)=\gamma$,
    \item  if $f_1(0)\neq0$ we have $\mu_2 g_1(0)=(\mu_2+\mu_3-1)g_1(m)$
    then for $\mu_2=1$ and $\mu_3=0$ gives $g_1(0)=0$. Then
        $\left\{
  \begin{array}{ll}
  g_1(m)=0 & \ \text{if} \ (\mu_2,\mu_3)\in\{(0,0),(1,1)\},\\
 g_1(m)=\gamma & \ \text{if} \ (\mu_2,\mu_3)\in\{(1,0),(0,1)\}.
  \end{array}\right.$
\qedhere
\end{enumerate}
\end{proof}
\begin{prop}\label{prop3:q neq 1,-1}
If $P_d$ is the averaging operator on $\mathcal{W}^{q}$ with degree $d$ satisfying $q^{d}=1$, $q^{2m}\neq 1$ and $f_2(0)\neq0$, then
$
\left\{
  \begin{array}{ll}
  f_1(m)=\gamma, & f_2(m)=f_2(0), \\
     g_1(m)=\frac{\gamma f_1(0)-\gamma^{2}}{f_2(0)}, & g_2(m)= f_1(0)-\gamma,
  \end{array}\right.
$
where $\gamma\in\mathbb{K}$.
\end{prop}
  \begin{proof}
Let $q^{d}=1,~~q^{2m}\neq 1$ and $f_2(0)\neq 0,$ as assumed.
Taking $n=0$ in \eqref{fifth eq} yields
\begin{align*}f_1(m)f_2(0)[m]=&f_1(m)f_2(m)[m+d]\\
=&f_1(m)f_2(m)[m]. \quad \text{(since $q^{d}=1$)}
\end{align*} Since $q^{2m}\neq1$ we have
$f_1(m)f_2(0)=f_1(m)f_2(m)$.
This together with $f_2(m)=\mu_1 f_2(0)$ gives $f_1(m)(\mu_1-1)=0$. Then
$
f_1(m)=\left\{
\begin{array}{lllll}
 0&if&\mu=0, \\
 \gamma &if& \mu_1=1.
\end{array}\right.
$
Taking $n=0$ in \eqref{tenth eq} yields
\begin{align*}
    f_2(m)g_2(0)[m]=&f_2(0)g_2(m)[-m]+f_2(m)g_2(m)[m+d]\\
    =&-f_2(0)g_2(m)[m]+f_2(m)g_2(m)[m]. \quad \text{(since $q^{d}=1$})
\end{align*}
Since $q^{2m}\neq 1$, we have
$f_2(m)g_2(0)=-f_2(0)g_2(m)+f_2(m)g_2(m)$.
This, with $f_2(m)=\mu_1 f_2(0)$, gives
$\mu_1 f_2(0)g_2(0)=-f_2(0)g_2(m)+\mu_1 f_2(0) g_2(m)$. Then
\begin{enumerate}
    \item if $\mu_1=0$ we have $g_2(m)=0$,
    \item if $\mu_1=1$ we have $g_2(0)=0$.
\end{enumerate}
Taking $n=0$ in \eqref{first eq} yields
\begin{align*}
    f_1(m)f_1(0)[m]=&f_1^{2}(m)[m+d]-g_1(m)f_2(m)[-m-d]\\
    =&f_1^{2}(m)[m]+g_1(m)f_2(m)[m].\quad \text{(since $q^{d}=1$)}
\end{align*}
Since $q^{2m}\neq 1$, we have
$f_1(m)f_1(0)=f_1^{2}(m)+g_1(m)f_2(m)$.
This, with $f_2(m)=\mu_1 f_2(0)$, $f_1(m)=\gamma$ and $\mu_1=1$, yields $g_1(m)=\frac{\gamma f_1(0)-\gamma^{2}}{f_2(0)}$.

Taking $n=0$ in \eqref{sixth eq} yields
\begin{align*}
    f_1(m)g_2(0)[m]=&f_2(0)g_1(m)[-m]+f_1(m)g_2(m)[m+d]\\
   =&-f_2(0)g_1(m)[m]+f_1(m)g_2(m)[m]. \quad \text{(since $q^{d}=1$)}
\end{align*}
Since $q^{2m}\neq0$ we have
$ f_1(m)g_2(0)=-f_2(0)g_1(m)+f_1(m)g_2(m)$.
This together with $g_2(0)=0$ and $f_1(m)=g_2(m)=0$ for $\mu_1=0$ gives $g_1(m)=0$.

Taking $n=0$ in the equation \eqref{fourteenth eq}, we have
\begin{align*}
    f_2(m)f_1(0)[m]=&f_2(m)f_1(m)[m+d]-g_2(m)f_2(m)[-m-d]\\
  =&f_2(m)f_1(m)[m]+g_2(m)f_2(m)[m]. \quad \text{(since $q^{d}\neq 1$)}
\end{align*}
Then
$ f_2(m)f_1(0)=f_2(m)f_1(m)+g_2(m)f_2(m)$.
This, together with $f_1(m)=\gamma$ for $\mu_1=1$, gives $g_2(m)=f_1(0)-\gamma$.
\end{proof}
\begin{thm}\label{classification 2}
Homogeneous averaging operators on the $q$-deformed $W(2,2)$ Hom-algebra $\mathcal{W}^{q}$ with degree $d$ such that $q^{d}=1$ and $q\neq -1,1$
must be one of the following operators, given for all $m\in \mathbb{Z}$, by
\begin{align*}
& \left\{
  \begin{array}{lllll}
 P_d^{1}(L_m)=\nu_1\delta_{q^{2m},1}L_{m+d}+(\nu_3\delta_{q^{2m},1}+\gamma)W_{m+d},\\
P_d^{1}(W_m)=\nu_2\delta_{q^{2m},1}L_{m+d}+\nu_4\delta_{q^{2m},1}W_{m+d},
  \end{array}\right.\\[0.3cm]
&\left\{
  \begin{array}{lllll}
 P_d^{2}(L_m)=(\nu_1\delta_{q^{2m},1}+\beta)L_{m+d}+(\nu_3\delta_{q^{2m},1}+\gamma)W_{m+d},\\
P_d^{2}(W_m)=\nu_2\delta_{q^{2m},1}L_{m+d}+\nu_4\delta_{q^{2m},1}W_{m+d},
  \end{array}\right.\\[0.3cm]
& \left\{
  \begin{array}{lllll}
 P_d^{3}(L_m)=(\nu_1\delta_{q^{2m},1}+\beta)L_{m+d}+\nu_3\delta_{q^{2m},1}W_{m+d},\\
P_d^{3}(W_m)=\nu_2\delta_{q^{2m},1}L_{m+d}+(\nu_4\delta_{q^{2m},1}+\beta)W_{m+d},
  \end{array}\right.\\[0.3cm]
&\left\{
  \begin{array}{l}
 P_d^{4}(L_m)=(\nu_1\delta_{q^{2m},1}+\gamma)L_{m+d}+\nu_3\delta_{q^{2m},1}W_{m+d},\\
P_d^{4}(W_m)=\nu_2\delta_{q^{2m},1}L_{m+d}+(\nu_4\delta_{q^{2m},1}+\beta)W_{m+d},
  \end{array}\right.\\[0.3cm]
& \left\{
  \begin{array}{lllll}
 P_d^{5}(L_m)=(\nu_1\delta_{q^{2m},1}+\gamma)L_{m+d}+(\nu_3\delta_{q^{2m},1}+\frac{\gamma \theta-\gamma^{2}}{\beta})W_{m+d},\\
P_d^{5}(W_m)=(\nu_2\delta_{q^{2m},1}+ \beta)L_{m+d}+(\nu_4\delta_{q^{2m},1}+\theta-\gamma)W_{m+d},
  \end{array}\right.
\end{align*}
where $\gamma,\theta,\nu_1,\nu_2,\nu_3,\nu_4\in\mathbb{K}$ and $\beta\in\mathbb{K}^{\ast}$.
\end{thm}
\begin{proof}
Directly by ombining Lemma \ref{lemma3} and Propositions\ref{prop1:q neq 1,-1}-\ref{prop3:q neq 1,-1}.
\end{proof}
\begin{thm}\label{thm:Leibniz1}
The homogeneous averaging operators on the $q$-deformed $W(2,2)$ Hom-algebra $\mathcal{W}^{q}$ with of degree $d$ such that $q^{d}= 1$ and $q\neq -1,1$ obtained in Theorem \ref{classification 2} provide the following Hom-Leibniz algebras
on the underlying linear space $\mathcal{W}^{q}:$
\begin{enumerate}
\item\label{inducedHomLeibnizWq:i1}
  $\{L_{m},L_{n}\}^{1} =\nu_1\delta_{q^{2m},1}[m-n]L_{m+n}+(\nu_3\delta_{q^{2m},1}+\gamma)[m-n]W_{m+n} \\
 \{L_{m},W_{n}\}^{1} =\nu_1\delta_{q^{2m},1}[m-n]W_{m+n}\\
 \{W_{m},L_{n}\}^{1} =\nu_2[m-n]\delta_{q^{2m},1}L_{m+n}+\nu_4\delta_{q^{2m},1}[m-n]W_{m+n}\\
  \{W_{m},W_{n}\}^{1} =\nu_2\delta_{q^{2m},1}[m-n]L_{m+n},$
\item
  $\{L_{m},L_{n}\}^{2} =(\nu_1\delta_{q^{2m},1}+\beta)[m-n]L_{m+n}+(\nu_3\delta_{q^{2m},1}+\gamma)[m-n]W_{m+n} \\
 \{L_{m},W_{n}\}^{2} =(\nu_1\delta_{q^{2m},1}+\beta)[m-n]W_{m+n}\\
 \{W_{m},L_{n}\}^{2} =\nu_2[m-n]\delta_{q^{2m},1}L_{m+n}+\nu_4\delta_{q^{2m},1}[m-n]W_{m+n}\\
  \{W_{m},W_{n}\}^{2} =\nu_2\delta_{q^{2m},1}[m-n]W_{m+n},$
\item
$\{L_{m},L_{n}\}^{3} =(\nu_1\delta_{q^{2m},1}+\beta)[m-n]L_{m+n}+\nu_3\delta_{q^{2m},1}[m-n]W_{m+n} \\
 \{L_{m},W_{n}\}^{3} =(\nu_1\delta_{q^{2m},1}+\beta)[m-n]W_{m+n}\\
 \{W_{m},L_{n}\}^{3} =\nu_2[m-n]\delta_{q^{2m},1}L_{m+n}+(\nu_4\delta_{q^{2m},1}+\beta)[m-n]W_{m+n}\\
  \{W_{m},W_{n}\}^{3} =\nu_2\delta_{q^{2m},1}[m-n]W_{m+n},$
 \item
 $\{L_{m},L_{n}\}^{4} =(\nu_1\delta_{q^{2m},1}+\gamma)[m-n]L_{m+n}+\nu_3\delta_{q^{2m},1}[m-n]W_{m+n} \\
 \{L_{m},W_{n}\}^{4} =(\nu_1\delta_{q^{2m},1}+\gamma)[m-n]W_{m+n}\\
 \{W_{m},L_{n}\}^{4} =\nu_2[m-n]\delta_{q^{2m},1}L_{m+n}+(\nu_4\delta_{q^{2m},1}+\beta)[m-n]W_{m+n}\\
  \{W_{m},W_{n}\}^{4} =\nu_2\delta_{q^{2m},1}[m-n]W_{m+n},$
   \item
   \label{inducedHomLeibnizWq:v1}
 $\{L_{m},L_{n}\}^{5} =(\nu_1\delta_{q^{2m},1}+\gamma)[m-n]L_{m+n}+(\nu_3\delta_{q^{2m},1}+\frac{\gamma\theta-\gamma^{2}}{\beta}[m-n]W_{m+n} \\
 \{L_{m},W_{n}\}^{5} =(\nu_1\delta_{q^{2m},1}+\gamma)[m-n]W_{m+n}\\
 \{W_{m},L_{n}\}^{5} =(\nu_2\delta_{q^{2m},1}+\beta)[m-n]L_{m+n}+(\nu_4\delta_{q^{2m},1}+\theta-\gamma)[m-n]W_{m+n}\\
  \{W_{m},W_{n}\}^{5} =(\nu_2\delta_{q^{2m},1}+\beta)[m-n]W_{m+n}$,
\end{enumerate}
  where $\nu_i,\gamma,\theta\in\mathbb{K}~and~\beta\in\mathbb{K}^{\ast}$.
\end{thm}
\begin{proof}

We demonstrate a proof of \ref{inducedHomLeibnizWq:i}. The others are proved analogously. For any
$m,n\in\mathbb{Z}$,
\begin{align*}
    &\{L_m,L_n\}^{1}=[P_d^{1}(L_m),L_n]= [\nu_1\delta_{q^{2m},1}L_{m+d}+(\nu_3\delta_{q^{2m},1}+\gamma)W_{m+d},L_n]\\
   &\quad=\nu_1[m+d-n]\delta_{q^{2m},1}L_{m+n+d}+(\nu_3\delta_{q^{2m},1}+\gamma)[m+d-n]W_{m+n+d}\\
    &\quad=\nu_1(m-n)\delta_{q^{2m},1}L_{m+n}+(\nu_3\delta_{q^{2m},1}+\gamma)[m-n]W_{m+n},\\
      &\{L_m,W_n\}^{1}=[P_d^{1}(L_m),W_n]= [\nu_1\delta_{q^{2m},1}L_{m+d}+(\nu_3\delta_{q^{2m},1}+\gamma)W_{m+d},W_n] \\
&\quad =\nu_1(m+d-n)\delta_{q^{2m},1}L_{m+n+d}, =\nu_1[m-n]\delta_{q^{2m},1}L_{m+n},\\
    &\{W_m,L_n\}^{1}=[P_d^{1}(W_m),L_n]= [\nu_1\delta_{q^{2m},1}L_{m+d}+\nu_4\delta_{q^{2m},1}W_{m+d},L_n]\\
   &\quad=\nu_2[m+d-n]\delta_{q^{2m},1}L_{m+n+d}+\nu_4\delta_{q^{2m},1}[m+d-n]W_{m+n+d}\\
      &\quad=\nu_2(m-n)\delta_{q^{2m},1}L_{m+n}+\nu_4\delta_{q^{2m},1}[m-n]W_{m+n},\\
           &\{W_m,W_n\}^{1}=[P_d^{1}(W_m),W_n]\\
   &\quad= [\nu_2\delta_{q^{2m},1}L_{m+d}+\nu_4\delta_{q^{2m},1}W_{m+d},W_n]=\nu_2(m+d-n)\delta_{q^{2m},1}W_{m+n+d}\\
 & \hspace{7.2cm} =\nu_2(m-n)\delta_{q^{2m},1}W_{m+n}.
 \qedhere
\end{align*}
\end{proof}
\begin{prop}
The Hom-Leibniz algebras $(\mathcal {W}^{q}, ~\{\cdot,\cdot\}^{i}, \alpha)$ for $i\in\{1,\cdots,5\}$ given in Theorem \ref{thm:Leibniz1} items \ref{inducedHomLeibnizWq:i1}-\ref{inducedHomLeibnizWq:v1} are respectively multiplicatives if and only if  \begin{enumerate}
\item\label{multiplic:i1}
 $q^2=-1$  \text{ or }$ ~\nu_1=\nu_2=\nu_3=\nu_4=\gamma=0;$
\item
 $q^2=-1$;
\item
   $q^2=-1$;
 \item
 $q^2=-1$;
   \item
   \label{multiplic:v}
 $q^2=-1$.
\end{enumerate}
\end{prop}
\begin{proof}
We prove \ref{multiplic:i1}, the others are proved analogously. For any
$m,n\in\mathbb{Z}$, we have
\begin{align*}
       &\alpha(\{L_m,L_n\}^{1})-\{\alpha(L_m),\alpha(L_n)\}^{1}\\
    &\quad=\alpha\Big(\nu_1\delta_{q^{2m},1}[m-n]L_{m+n}+(\nu_3\delta_{q^{2m},1}+\gamma)[m-n]W_{m+n}\Big)\\&\quad\quad-\{(q^m+q^{-m})L_m,(q^n+q^{-n})L_n\}^{1}\\
    &\quad=\nu_1\delta_{q^{2m},1}[m-n](q^{m+n}+q^{-m-n})L_{m+n}\\&\quad\quad+(\nu_3\delta_{q^{2m},1}+\gamma)[m-n](q^{m+n}+q^{-m-n})W_{m+n}\\
    &\quad\quad-(q^m+q^{-m})(q^n+q^{-n})\Big((\nu_1\delta_{q^{2m},1}+\beta)[m-n]L_{m+n}\\&\quad\quad+(\nu_3\delta_{q^{2m},1}+\gamma)[m-n]W_{m+n}\Big)\\
    &\quad=\nu_1\delta_{q^{2m},1}[m-n](q^{m-n}+q^{n-m})L_{m+n}+(\nu_3\delta_{q^{2m},1}+\gamma)[m-n](q^{m-n}+q^{n-m})W_{m+n}\\
    &\quad =   -(q^{m-n}-q^{n-m})(q^{m-n}+q^{n-m})\big(\frac{\nu_1\delta_{q^{2m},1}}{q-q^{-1}}L_{m+n}-\frac{\nu_3\delta_{q^{2m},1}+\gamma}{q-q^{-1}}W_{m+n}\big) \\
    &\quad=  -(q^{2(m-n)}-q^{2(n-m)})\big(\frac{\nu_1\delta_{q^{2m},1}}{q-q^{-1}}L_{m+n}+\frac{\nu_3\delta_{q^{2m},1}+\gamma}{q-q^{-1}}W_{m+n}\big), \\
     & \alpha(\{L_m,W_n\}^{1})-\{\alpha(L_m),\alpha(W_n)\}^{1}\\
     &\quad=\alpha\Big((\nu_1\delta_{q^{2m},1})[m-n]W_{m+n}-\{(q^m+q^{-m})L_m,(q^n+q^{-n})W_n\}^{2}\\
    &\quad=\nu_1\delta_{q^{2m},1}[m-n](q^{m+n}+q^{-m-n})W_{m+n}-\nu_1\delta_{q^{2m},1}[m-n](q^{m}+q^{-m})(q^{n}+q^{-n})W_{m+n}\\
    &\quad =  -(q^{m-n}-q^{n-m})(q^{m-n}+q^{n-m})\big(\frac{\nu_1\delta_{q^{2m},1}}{q-q^{-1}}W_{m+n}\big)\\
       &\quad =     -(q^{2(m-n)}-q^{2(n-m)})\big(\frac{\nu_1\delta_{q^{2m},1}}{q-q^{-1}}W_{m+n}\big), \\
       &\alpha(\{W_m,L_n\}^{1})-\{\alpha(W_m),\alpha(L_n)\}^{1}\\
    &\quad=\alpha\Big(\nu_2\delta_{q^{2m},1}[m-n]L_{m+n}+\nu_4\delta_{q^{2m},1}[m-n]W_{m+n}\Big)\\&\quad\quad-\{(q^m+q^{-m})W_m,(q^n+q^{-n})L_n\}^{3}\\
    &\quad=\nu_2\delta_{q^{2m},1}[m-n](q^{m+n}+q^{-m-n})L_{m+n}\\&\quad\quad+\nu_4\delta_{q^{2m},1}[m-n](q^{m+n}+q^{-m-n})W_{m+n}\\
    &\quad\quad-(q^m+q^{-m})(q^n+q^{-n})\Big((\nu_2\delta_{q^{2m},1}+\beta)[m-n]L_{m+n}\\&\quad\quad+\nu_4\delta_{q^{2m},1}[m-n]W_{m+n}\Big)\\
    &\quad=\nu_2\delta_{q^{2m},1}[m-n](q^{m-n}+q^{n-m})L_{m+n}+\nu_4\delta_{q^{2m},1}[m-n](q^{m-n}+q^{n-m})W_{m+n}\\
    &\quad =   -(q^{m-n}-q^{n-m})(q^{m-n}+q^{n-m})\big(\frac{\nu_2\delta_{q^{2m},1}}{q-q^{-1}}L_{m+n}-\frac{\nu_4\delta_{q^{2m},1}}{q-q^{-1}}W_{m+n}\big)\\
    &\quad=   -(q^{2(m-n)}-q^{2(n-m)})\big(\frac{\nu_2\delta_{q^{2m},1}}{q-q^{-1}}L_{m+n}+\frac{\nu_4\delta_{q^{2m},1}}{q-q^{-1}}W_{m+n}\big),\\
       & \alpha(\{W_m,W_n\}^{1})-\{\alpha(W_m),\alpha(W_n)\}^{1}\\
     &\quad=\alpha\Big((\nu_2\delta_{q^{2m},1})[m-n]W_{m+n}-\{(q^m+q^{-m})L_m,(q^n+q^{-n})W_n\}^{4}\\
    &\quad=\nu_2\delta_{q^{2m},1}[m-n](q^{m+n}+q^{-m-n})L_{m+n}-\nu_2\delta_{q^{2m},1}[m-n](q^{m}+q^{-m})(q^{n}+q^{-n})L_{m+n}\\
    &\quad = -(q^{m-n}-q^{n-m})(q^{m-n}+q^{n-m})\big(\frac{\nu_2\delta_{q^{2m},1}}{q-q^{-1}}L_{m+n}\big)\\
          &\quad =     -(q^{2(m-n)}-q^{2(n-m)})\big(\frac{\nu_2\delta_{q^{2m},1}}{q-q^{-1}}L_{m+n}\big).
\end{align*}
So, by Theorem \ref{thm:multnonmultcnds} \ref{thm:i:multnonmultcnds},
\begin{align*}
& \text{$(\mathcal {W}^{q}, \{\cdot,\cdot\}^1, \alpha)$ is multiplicative}\  \Leftrightarrow \\
& \forall\ m,n\in\mathbb{Z}:\
\left\{\begin{array}{llll} q^{2(m-n)}-q^{2(n-m)}=0, or\\ \nu_1 \delta_{q^{2m},1}=\nu_3 \delta_{q^{2m},1}+\gamma=\\\nu_1 \delta_{q^{2m},1}=\nu_2 \delta_{q^{2m},1}=\nu_4 \delta_{q^{2m},1}=0,
\end{array}\right.,\  \Leftrightarrow \\[0.2cm]
& \forall\ m,n\in\mathbb{Z}:\
\left\{\begin{array}{llll} q^{4(m-n)}=0, or\\ \nu_1=\nu_2=\nu_3=\nu_4=\gamma=0,
\end{array}\right.,
\Leftrightarrow\\[0.2cm]
& \quad\quad\quad\quad\quad\quad
\left\{\begin{array}{llll} \forall p\in\mathbb{Z},~q^{4p}=0, or\\ \nu_1=\nu_2=\nu_3=\nu_4=\gamma=0,
\end{array}\right.,\Leftrightarrow\\[0.2cm]
&\quad\quad\quad\quad \quad\quad\left\{\begin{array}{llll}q^{4}=1, or\\ \nu_1=\nu_2=\nu_3=\nu_4=\gamma=0,
\end{array}\right.,
\Leftrightarrow\\[0.2cm]
&\quad\quad\quad\quad \quad\quad\left\{\begin{array}{llll} q^{2}=-1, or\\\nu_1=\nu_2=\nu_3=\nu_4=\gamma=0,
\end{array}\right.
\Leftrightarrow\\[0.2cm]
& q=\pm i \ \text{if}\ \exists \ i\in \mathbb{K}: \ i^2=-1 \
\text{(for example if $\mathbb{K}$ is algebraically closed)}, \\& or~~\nu_1=\nu_2=\nu_3=\nu_4=\gamma=0.
\qedhere
\end{align*}
\end{proof}
\subsection*{Case 2: $(q,d)\in \{1\}\times\mathbb{Z}\cup \{-1\}\times 2\mathbb{Z}$}
\begin{rmk}
The equations \eqref{first eq}-\eqref{fourteenth eq} are equivalents for $(q,d)\in \{1\}\times\mathbb{Z}$ and for $(q,d)\in \{-1\}\times 2\mathbb{Z}$.
\end{rmk}
\subsubsection*{Subcase 1: $m=0$ and $d=0$}
  Take $n=0$ in \eqref{first eq}-\eqref{fourteenth eq}. For $d=0$ and $m=0$, the functions $f_1,~f_2,~g_1$ and $g_2$ satisfy
$$
f_1(0)=\nu_1,\ f_2(0)=\nu_2,\ g_1(0)=\nu_3,\ g_2(0)=\nu_4,\ \nu_i\in\mathbb{K}.
$$
Then we have the following Proposition.
\begin{prop}\label{prop1:q=1,-1}
If $P_0$ is an averaging operator on $\mathcal{W}^{q}$ with degree $d=0$, then
$$
\left\{
  \begin{array}{ll}
  f_1(0)=\nu_1, &
    f_2(0)=\nu_2,\\
     g_1(0)=\nu_3,&
     g_2(0)=\nu_4,
  \end{array}\right.
\quad \text{where $\nu_i\in\mathbb{K}.$} $$
  \end{prop}
\subsubsection*{Subcase 1: $m\neq0$ and $d=0$}
Taking $n=0$ in \eqref{nineth eq} we give
$f_2(m)(f_2(0)-f_2(m))=0$. Hence,
\begin{equation*}
f_2(m)=\mu_1 f_2(0),\quad\mu_1\in\{0,1\}.
\end{equation*}
Then, we have the following Proposition.
\begin{prop}\label{prop2:q=1,-1}
If $P_0$ is an averaging operators on $\mathcal{W}^{q}$ with degree $d=0$ such that $m\neq 0$ and $f_2(0)=0,$
then
\begin{enumerate}
    \item if $f_1(0)=0$, then
$f_1(m)=0,\
  f_2(m)=0,\
     g_1(m)=\gamma,\
  g_2(m)= 0,
$
where $\gamma\in\mathbb{K}$;
  \item  if $f_1(0)\neq 0$, then
\begin{enumerate}
  \item
 $ f_1(m)=f_1(0),\
  f_2(m)=0,\
     g_1(m)=\gamma,\
  g_2(m)= 0$, where $\gamma\in\mathbb{K};$
  \item
  $
  f_1(m)=f_1(0),\
  f_2(m)=0,\
     g_1(m)=0,\
  g_2(m)=f_1(0)$, where $\gamma\in\mathbb{K};$
  \item
  $
  f_1(m)=0,\
  f_2(m)=0,\
     g_1(m)=\gamma,\
  g_2(m)= f_1(0)$, where $\gamma\in\mathbb{K}.$
\end{enumerate}
  \end{enumerate}
\end{prop}

\begin{proof}
Let $m\neq 0$, $d=0$ and $f_2(0)=0$ as assumed.
Taking $n=0$ in \eqref{first eq} we obtain
 $f_1(m)(f_1(0)-f_1(m))=0$. Thus,
$f_1(m)=\mu_2 f_1(0)$, $\mu_2\in\{0,1\}$. Setting $n=0$ in \eqref{thirteenth eq} yields
\begin{align*}
mf_2(m)g_1(0)=&-mg_2(m)f_1(0)+mf_2(m)g_1(m)+mg_2^{2}(m).
\end{align*}
Since $m\neq 0$, we have $f_2(m)g_1(0)=g_2(m)f_1(0)-f_2(m)g_1(m)-g_2^{2}(m)$, from which together with $f_2(m)=0$, we get
$$g_2(m)=\mu_3 f_1(0), \quad \mu_3\in\{0,1\}.$$
Taking $n=0$ in \eqref{second eq} and
from $m\neq 0$ we get $f_1(m)g_1(0)=-f_1(0)g_1(m)+f_1(m)g_1(m)+g_1(m)g_2(m)$, and with $f_1(m)=\mu_2 f_1(0)$ and $g_2(m)=\mu_3 f_1(0)$, we obtain $g_1(m)(\mu_2+\mu_3-1)f_1(0)=\mu_2 f_1(0)g_1(0)$. Then, we have the two cases:
\begin{enumerate}
    \item
    if $f_1(0)=0$, then
    $g_1(m)=\gamma$,
    \item  if $f_1(0)\neq0$ we have $\mu_2 g_1(0)=(\mu_2+\mu_3-1)g_1(m)$. Then, \\
        $\left\{
  \begin{array}{ll}
  g_1(m)=0 & \ \text{if} \ (\mu_2,\mu_3)\in\{(0,0),(1,1)\},\\
 g_1(m)=\gamma & \ \text{if} \ (\mu_2,\mu_3)\in\{(1,0),(0,1)\}.
  \end{array}\right.$
\qedhere
\end{enumerate}
\end{proof}
\begin{prop}\label{prop3:q=1,-1}
If $P_d$ is the averaging operator on $\mathcal{W}^{q}$ with degree $d=0$ such that $m\neq 0$ and $f_2(0)\neq0$, then
$
\left\{
  \begin{array}{ll}
  f_1(m)=\gamma, & f_2(m)=f_2(0), \\
     g_1(m)=\frac{\gamma f_1(0)-\gamma^{2}}{f_2(0)}, & g_2(m)= f_1(0)-\gamma,
  \end{array}\right.
$
where $\gamma\in\mathbb{K}$.
\end{prop}
  \begin{proof}
Let $d=0,~~m\neq 0$ and $f_2(0)\neq 0,$ as assumed.
Taking $n=0$ in \eqref{fifth eq} yields
$mf_1(m)f_2(0)=mf_1(m)f_2(m)$. Since $m\neq 0$ we have
$f_1(m)f_2(0)=f_1(m)f_2(m)$.
This together with $f_2(m)=\mu_1 f_2(0)$ gives $f_1(m)(\mu_1-1)=0$. Then
$
f_1(m)=\left\{
\begin{array}{lllll}
 0&if&\mu=0, \\
 \gamma &if& \mu_1=1.
\end{array}\right.
$
Taking $n=0$ in \eqref{tenth eq} yields
   $ mf_2(m)g_2(0)=-mf_2(0)g_2(m)+mf_2(m)g_2(m).$
Since $m\neq 0$, we have
$f_2(m)g_2(0)=-f_2(0)g_2(m)+f_2(m)g_2(m)$.
This, with $f_2(m)=\mu_1 f_2(0)$, gives
$\mu_1 f_2(0)g_2(0)=-f_2(0)g_2(m)+\mu_1 f_2(0) g_2(m)$. Then
\begin{enumerate}
    \item if $\mu_1=0$ we have $g_2(m)=0$,
    \item if $\mu_1=1$ we have $g_2(0)=0$.
\end{enumerate}
Taking $n=0$ in \eqref{first eq} yields
   $ f_1(m)f_1(0)=mf_1^{2}(m)+mg_1(m)f_2(m).$
Since $m\neq 0$, we have
$f_1(m)f_1(0)=f_1^{2}(m)+g_1(m)f_2(m)$.
This, with $f_2(m)=\mu_1 f_2(0)$, $f_1(m)=\gamma$ and $\mu_1=1$, yields $g_1(m)=\frac{\gamma f_1(0)-\gamma^{2}}{f_2(0)}$.

Taking $n=0$ in \eqref{sixth eq} yields
$
    mf_1(m)g_2(0)=-mf_2(0)g_1(m)+mf_1(m)g_2(m).$
Since $m\neq0$ we have
$ f_1(m)g_2(0)=-f_2(0)g_1(m)+f_1(m)g_2(m)$.
This together with $g_2(0)=0$ and $f_1(m)=g_2(m)=0$ for $\mu_1=0$ gives $g_1(m)=0$.

Taking $n=0$ in the equation \eqref{fourteenth eq} yields
$
    mf_2(m)f_1(0)=mf_2(m)f_1(m)+mg_2(m)f_2(m).
$
Then
$ f_2(m)f_1(0)=f_2(m)f_1(m)+g_2(m)f_2(m)$.
This, together with $f_1(m)=\gamma$ for $\mu_1=1$, gives $g_2(m)=f_1(0)-\gamma$.
\end{proof}

\begin{thm}\label{classification 2:q-1,-1}
The Homogeneous averaging operators on the $q$-deformed $W(2,2)$ Hom-algebra $\mathcal{W}^{q}$ with degree $d=0$.
must be one of the following operators, given for all $m\in \mathbb{Z}$, by
\begin{align*}
& \left\{
  \begin{array}{lllll}
 P_0^{1}(L_m)=\nu_1\delta_{m,0}L_{m}+(\nu_3\delta_{m,0}+\gamma)W_{m},\\
P_0^{1}(W_m)=\nu_2\delta_{m,0}L_{m}+\nu_4\delta_{m,0}W_{m},
  \end{array}\right.\\[0.3cm]
&\left\{
  \begin{array}{lllll}
 P_0^{2}(L_m)=(\nu_1\delta_{m,0}+\beta)L_{m}+(\nu_3\delta_{m,0}+\gamma)W_{m},\\
P_0^{2}(W_m)=\nu_2\delta_{m,0}L_{m}+\nu_4\delta_{m,0}W_{m},
  \end{array}\right.\\[0.3cm]
& \left\{
  \begin{array}{lllll}
 P_0^{3}(L_m)=(\nu_1\delta_{m,0}+\beta)L_{m}+\nu_3\delta_{m,0}W_{m},\\
P_0^{3}(W_m)=\nu_2\delta_{m,0}L_{m}+(\nu_4\delta_{m,0}+\beta)W_{m},
  \end{array}\right.\\[0.3cm]
&\left\{
  \begin{array}{l}
 P_0^{4}(L_m)=(\nu_1\delta_{m,0}+\gamma)L_{m}+\nu_3\delta_{m,0}W_{m}\\
P_0^{4}(W_m)=\nu_2\delta_{m,0}L_{m}+(\nu_4\delta_{m,0}+\beta)W_{m},
  \end{array}\right.\\[0.3cm]
& \left\{
  \begin{array}{lllll}
 P_0^{5}(L_m)=(\nu_1\delta_{m,0}+\gamma)L_{m}+(\nu_3\delta_{m,0}+\frac{\gamma \theta-\gamma^{2}}{\beta})W_{m},\\
P_0^{5}(W_m)=(\nu_2\delta_{m,0}+ \beta)L_{m}+(\nu_4\delta_{m,0}+\theta-\gamma)W_{m}.
  \end{array}\right.
\end{align*}
where $\gamma,\theta,\nu_1,\nu_2,\nu_3,\nu_4\in\mathbb{K}$ and $\beta\in\mathbb{K}^{\ast}$.
\end{thm}
\begin{proof}
Directly by ombining Lemma \ref{lemma3} and Propositions \ref{prop1:q=1,-1}-\ref{prop3:q=1,-1}.
\end{proof}
\begin{thm}\label{thm:Leibniz}
The homogeneous averaging operators on the $q$-deformed $W(2,2)$ Hom-algebra $\mathcal{W}^{q}$ with of degree $d=0$ obtained in Theorem \ref{classification 2} provide the following Hom-Leibniz algebras
on the underlying linear space $\mathcal{W}^{q}:$
\begin{enumerate}
\item\label{inducedHomLeibnizWq:i}
  $\{L_{m},L_{n}\}^{1} =\nu_1\delta_{m,0}[m-n]L_{m+n}+(\nu_3\delta_{m,0}+\gamma)[m-n]W_{m+n} \\
 \{L_{m},W_{n}\}^{1} =\nu_1\delta_{m,0}[m-n]W_{m+n}\\
 \{W_{m},L_{n}\}^{1} =\nu_2[m-n]\delta_{m,0}L_{m+n}+\nu_4\delta_{m,0}[m-n]W_{m+n}\\
  \{W_{m},W_{n}\}^{1} =\nu_2\delta_{m,0}[m-n]L_{m+n},$
\item
  $\{L_{m},L_{n}\}^{2} =(\nu_1\delta_{m,0}+\beta)[m-n]L_{m+n}+(\nu_3\delta_{m,0}+\gamma)[m-n]W_{m+n} \\
 \{L_{m},W_{n}\}^{2} =(\nu_1\delta_{m,0}+\beta)[m-n]W_{m+n}\\
 \{W_{m},L_{n}\}^{2} =\nu_2[m-n]\delta_{m,0}L_{m+n}+\nu_4\delta_{m,0}[m-n]W_{m+n}\\
  \{W_{m},W_{n}\}^{2} =\nu_2\delta_{m,0}[m-n]W_{m+n},$
\item
$\{L_{m},L_{n}\}^{3} =(\nu_1\delta_{m,0}+\beta)[m-n]L_{m+n}+\nu_3\delta_{m,0}[m-n]W_{m+n} \\
 \{L_{m},W_{n}\}^{3} =(\nu_1\delta_{m,0}+\beta)[m-n]W_{m+n}\\
 \{W_{m},L_{n}\}^{3} =\nu_2[m-n]\delta_{m,0}L_{m+n}+(\nu_4\delta_{m,0}+\beta)[m-n]W_{m+n}\\
  \{W_{m},W_{n}\}^{3} =\nu_2\delta_{m,0}[m-n]W_{m+n},$
 \item
 $\{L_{m},L_{n}\}^{4} =(\nu_1\delta_{m,0}+\gamma)[m-n]L_{m+n}+\nu_3\delta_{m,0}[m-n]W_{m+n} \\
 \{L_{m},W_{n}\}^{4} =(\nu_1\delta_{m,0}+\gamma)[m-n]W_{m+n}\\
 \{W_{m},L_{n}\}^{4} =\nu_2[m-n]\delta_{m,0}L_{m+n}+(\nu_4\delta_{m,0}+\beta)[m-n]W_{m+n}\\
  \{W_{m},W_{n}\}^{4} =\nu_2\delta_{m,0}[m-n]W_{m+n},$
   \item
   \label{inducedHomLeibnizWq:v}
 $\{L_{m},L_{n}\}^{5} =(\nu_1\delta_{m,0}+\gamma)[m-n]L_{m+n}+(\nu_3\delta_{m,0}+\frac{\gamma\theta-\gamma^{2}}{\beta}[m-n]W_{m+n} \\
 \{L_{m},W_{n}\}^{5} =(\nu_1\delta_{m,0}+\gamma)[m-n]W_{m+n}\\
 \{W_{m},L_{n}\}^{5} =(\nu_2\delta_{m,0}+\beta)[m-n]L_{m+n}+(\nu_4\delta_{m,0}+\theta-\gamma)[m-n]W_{m+n}\\
  \{W_{m},W_{n}\}^{5} =(\nu_2\delta_{m,0}+\beta)[m-n]W_{m+n}$,
\end{enumerate}
  where $\nu_i,\gamma,\theta\in\mathbb{K}~and~\beta\in\mathbb{K}^{\ast}$.
\end{thm}
\begin{proof}
We demonstrate a proof of \ref{inducedHomLeibnizWq:i}. The others are proved analogously. For any
$m,n\in\mathbb{Z}$,
\begin{align*}
    &\{L_m,L_n\}^{1}=[P_0^{1}(L_m),L_n]= [\nu_1\delta_{m,0}L_{m}+(\nu_3\delta_{m,0}+\gamma)W_{m},L_n]\\
   &\quad=\nu_1[m-n]\delta_{m,0}L_{m+n}+(\nu_3\delta_{m,0}+\gamma)[m-n]W_{m+n},\\
      &\{L_m,W_n\}^{1}=[P_d^{1}(L_m),W_n]= [\nu_1\delta_{m,0}L_{m}+(\nu_3\delta_{m,0}+\gamma)W_{m},W_n] \\
&\quad =\nu_1[m-n]\delta_{m,0}L_{m+n},\\
    &\{W_m,L_n\}^{1}=[P_0^{1}(W_m),L_n]= [\nu_1\delta_{m,0}L_{m}+\nu_4\delta_{m,0}W_{m},L_n]\\
      &\quad=\nu_2(m-n)\delta_{m,0}L_{m+n}+\nu_4\delta_{m,0}[m-n]W_{m+n},\\
           &\{W_m,W_n\}^{1}=[P_0^{1}(W_m),W_n]\\
   &\quad= [\nu_2\delta_{m,0}L_{m}+\nu_4\delta_{m,0}W_{m},W_n]=\nu_2(m-n)\delta_{m,0}W_{m+n}.
\qedhere
\end{align*}
\end{proof}

\begin{prop}
The Hom-Leibniz algebras $(\mathcal {W}^{q}, ~\{\cdot,\cdot\}^{i}, \alpha)$ induced by $P_{0}^{i}$ for all $i\in\{1,\dots, 5\}$ is multiplicative if and only if $i=1$ and $\nu_1=\nu_2=\nu_3=\nu_4=\gamma=0.$
\end{prop}
\begin{proof}
For any
$m,n\in\mathbb{Z}$, we have
\begin{align*}
       &\alpha(\{L_m,L_n\}^{1})-\{\alpha(L_m),\alpha(L_n)\}^{1}\\
    &\quad=\alpha\Big(\nu_1\delta_{m,0}[m-n]L_{m+n}+(\nu_3\delta_{m,0}+\gamma)[m-n]W_{m+n}\Big)\\&\quad\quad-\{2 q^{m}L_m,2q^{N}L_n\}^{1}\\
    &\quad=-2q^{m+n}[m-n](\nu_1\delta_{m,0}L_{m+n}+(\nu_3\delta_{m,0}+\gamma)W_{m+n})\\
     & \alpha(\{L_m,W_n\}^{1})-\{\alpha(L_m),\alpha(W_n)\}^{1}\\
     &\quad=\alpha\Big((\nu_1\delta_{m,0})[m-n]W_{m+n}-\{2q^m L_m,2q^n W_n\}^{1}\\
    &\quad=2 q^{m+n}\nu_1\delta_{m,0}[m-n]W_{m+n}-4 q^{m+n}\nu_1\delta_{m,0}[m-n]W_{m+n}\\
    &\quad=-2 q^{m+n}\nu_1\delta_{m,0}[m-n]W_{m+n}\\
       &\alpha(\{W_m,L_n\}^{1})-\{\alpha(W_m),\alpha(L_n)\}^{1}\\
    &\quad=2 q^{m+n}[m-n](\nu_2\delta_{m,0}L_{m+n}+\nu_4 \delta_{m,0} W_{m+n}) -4 q^{m+n}[m-n](\nu_2\delta_{m,0}L_{m+n}+\nu_4 \delta_{m,0} W_{m+n})\\
       & \alpha(\{W_m,W_n\}^{1})-\{\alpha(W_m),\alpha(W_n)\}^{1}\\
     &\quad=-2 q^{m+n}\nu_2\delta_{m,0}[m-n] L_{m+n}.
\end{align*}
So, by Theorem \ref{thm:multnonmultcnds} \ref{thm:i:multnonmultcnds},

\text{$(\mathcal {W}^{q}, \{\cdot,\cdot\}^1, \alpha)$ is multiplicative}\  $\Leftrightarrow \forall m\in\mathbb{Z},~
\nu_1\delta_{m,0}=\nu_3 \delta_{m,0}+\gamma=\nu_2 \delta_{m,0}=\nu_4 \delta_{m,0}=0\Leftrightarrow \nu_1=\nu_2=\nu_3=\nu_4=\gamma=0.$

Similarly, for all $i\in\{2,\dots,5\}$ we prove that the Hom-Leibniz algebras $(\mathcal {W}^{q}, \{\cdot,\cdot\}^i, \alpha)$ are not multiplicative.
\qedhere
\end{proof}


\begin{thebibliography}{99}
\bibitem{AbdaouiAmmarMakhloufCohhomLiecolalg2015}
Abdaoui, K., Ammar, F., Makhlouf, A.: Constructions and cohomology of Hom-Lie color algebras, Comm. Algebra, \textbf{43}, 4581-4612 (2015)
\bibitem{belhsine}
Abdelkader, B. H.: Generalized derivations of BiHom-Lie algebras,
J. Gen. Lie Theory Appl. \textbf{11}(1), 1-7 (2017)
\bibitem{AbramovSilvestrov:3homLiealgsigmaderivINvol}
Abramov, V., Silvestrov, S.: $3$-Hom-Lie algebras based on $\sigma$-derivation and involution, Adv. Appl. Clifford Algebras, \textbf{30}(45) (2020)
\bibitem{Ag} Aguiar, M.:  Pre-Poisson algebras, Lett. Math. Phys. {\bf 54} 263-277 (2000)
\bibitem{AizawaSaito}
Aizawa, N., Sato, H.: $q$-Deformation of the Virasoro algebra with central extension, Phys. Lett. B, \textbf{256}, 185-190 (1991) (Hiroshima University preprint, preprint HUPD-9012 (1990))
\bibitem{AmmarEjbehiMakhlouf:homdeformation}
Ammar, F., Ejbehi, Z., Makhlouf, A.: Cohomology and deformations of Hom-algebras, J. Lie Theory, \textbf{21}(4), 813-836 (2011)
\bibitem{AmmarMabroukMakhloufCohomnaryHNLalg2011}
Ammar, F., Mabrouk, S., Makhlouf, A.: Representations and cohomology of $n$-ary multiplicative Hom-Nambu-Lie algebras, J. Geom. Phys. \textbf{61}(10), 1898-1913 (2011) 
\bibitem{AmmarMakhloufHomLieSupAlg2010}
Ammar, F., Makhlouf, A.: Hom-Lie superalgebras and Hom-Lie admissible superalgebras, J. Algebra, \textbf{324}(7), 1513-1528 (2010)
\bibitem{AmmarMakhloufSaadaoui2013:CohlgHomLiesupqdefWittSup}
Ammar, F., Makhlouf A., Saadaoui, N.: Cohomology of Hom-Lie superalgebras and $q$-deformed Witt superalgebra, Czechoslovak Math. J. \textbf{68}, 721-761 (2013)
\bibitem{AmmarMakhloufSilv:TernaryqVirasoroHomNambuLie}
Ammar, F., Makhlouf, A., Silvestrov, S.: Ternary $q$-Virasoro-Witt Hom-Nambu-Lie algebras, J. Phys. A: Math. Theor. \textbf{43}(26), 265204 (2010)
\bibitem{ArmakanFarhangdoost:IJGMMP}
Armakan A., Farhangdoost, M. R.: Geometric aspects of extensions of Hom-Lie superalgebras, Int. J. Geom. Methods Mod. Phys. \textbf{14}, 1750085 (2017)
\bibitem{ArmakanSilv:NondegKillingformsHomLiesuperalg}
Armakan A., Farhangdoost, M. R., Silvestrov S.: Non-degenerate Killing forms on Hom-Lie superalgebras, arXiv:2010.01778v2 [math.RA] (2020)
\bibitem{ArmakanSilv:envelalgcertaintypescolorHomLie}
Armakan A., Silvestrov S.: Enveloping algebras of certain types of color Hom-Lie algebras. In: Silvestrov, S., Malyarenko, A., Ran\u{c}i\'{c}, M. (Eds.), Algebraic Structures and Applications, Springer Proceedings in Mathematics and Statistics, \textbf{317}, Springer, Ch. 10, 257-284 (2020)
\bibitem{ArmakanSilvFarh:envelopalgcolhomLiealg}
Armakan, A., Silvestrov, S., Farhangdoost, M. R.: Enveloping algebras of color Hom-Lie algebras, Turk. J. Math. \textbf{43}, 316-339 (2019)
(arXiv:1709.06164 [math.QA] (2017)) 
\bibitem{ArmakanSilvFarh:exthomLiecoloralg}
Armakan, A., Silvestrov, S., Farhangdoost, M. R.: Extensions of Hom-Lie color algebras, Georgian Math. J. (2019),  doi:10.1515/gmj-2019-2033, (arXiv:1709.08620 [math.QA] (2017))
\bibitem{akms:ternary}
Arnlind, J., Kitouni, A., Makhlouf, A., Silvestrov, S.:
Structure and cohomology of $3$-Lie algebras induced by Lie algebras, In: Makhlouf, A., Paal, E., Silvestrov, S. D., Stolin, A., Algebra, Geometry and Mathematical Physics, Springer Proceedings in Mathematics and Statistics, \textbf{85}, Springer, 123-144 (2014)
\bibitem{ams:ternary}
Arnlind, J., Makhlouf, A., Silvestrov, S.:
Ternary Hom-Nambu-Lie algebras induced by Hom-Lie algebras, J. Math. Phys. \textbf{51}(4), 043515 (2010)
\bibitem{ArnlindMakhloufSilvnaryHomLieNambuJMP2011}
Arnlind, J., Makhlouf, A. Silvestrov, S.: Construction of $n$-Lie algebras and $n$-ary Hom-Nambu-Lie algebras, J. Math. Phys. \textbf{52}(12), 123502 (2011)
\bibitem{Bakayoko2014:ModulescolorHomPoisson}
Bakayoko, I.: Modules over color Hom-Poisson  algebras, J. Gen. Lie  Theory Appl. \textbf{8}(1), 1000212 (2014) 
\bibitem{BakayokoDialo2015:genHomalgebrastr} Bakayoko, I., Diallo, O. W.: Some generalized Hom-algebra structures, J. Gen. Lie Theory Appl.
\textbf{9}(1), 1000226 (2015)
\bibitem{BakyokoSilvestrov:Homleftsymmetriccolordialgebras}
Bakayoko, I., Silvestrov, S.: Hom-left-symmetric color dialgebras, Hom-triden\-driform color algebras and Yau's twisting generalizations,
arXiv:1912.01441 [math.RA] (2019)
\bibitem{BakyokoSilvestrov:MultiplicnHomLiecoloralg}
Bakayoko, I., Silvestrov, S.: Multiplicative $n$-Hom-Lie color algebras,
In: Silvestrov, S., Malyarenko, A., Ran\u{c}i\'{c}, M. (Eds.), Algebraic Structures and Applications, Springer Proceedings in Mathematics and Statistics, \textbf{317}, Springer, Ch. 7, 159-187 (2020). (arXiv:1912.10216[math.QA])
\bibitem{BakayokoToure2019:genHomalgebrastr} Bakayoko, I., Tour\'e, B. M.: Constructing Hom-Poisson color algebras, Int. J. Algebra, \textbf{13}(1), 1-16 (2019)
\bibitem{BenHassineChtiouiMabroukNcib:Strcohom3LieRinehartsuperalg}
Ben Hassine, A., Chtioui, T., Mabrouk S., Silvestrov, S.: Structure and cohomology of $3$-Lie-Rinehart superalgebras, arXiv:2010.01237 [math.RA] (2020)
\bibitem{BenHassineMabroukNcib:ConstrMultiplicnaryhomNambualg}
Ben Hassine, A., Mabrouk S., Ncib, O.: Some Constructions of Multiplicative $n$-ary
    hom-Nambu Algebras, Adv. Appl. Clifford Algebras, \textbf{29}(88) (2019) 
\bibitem{BenMakh:Hombiliform}
Benayadi, S., Makhlouf, A.: Hom-Lie algebras with symmetric invariant nondegenerate bilinear forms, J. Geom. Phys. \textbf{76}, 38-60 (2014)
\bibitem{Bi} G. Birkhoff, Moyennes de fonctions born\'ees, Coil. Internat. Centre Nat. Recherthe Sci. (Paris), {\em Alg\`ebre Th\'eforie Nombres} {\bf 24} (1949), 149-153.
\bibitem{CaoChen2012:SplitregularhomLiecoloralg}
Cao, Y., Chen, L.: On split regular Hom-Lie color algebras, Comm. Algebra \textbf{40}, 575-592 (2012)
\bibitem{ChaiElinPop}
Chaichian, M., Ellinas, D., Popowicz, Z.: Quantum conformal algebra with central extension, Phys. Lett. B, \textbf{248}, 95-99 (1990)
\bibitem{ChaiIsLukPopPresn}
Chaichian, M., Isaev, A. P., Lukierski, J., Popowic, Z., Pre\v{s}najder, P.: $q$-Defor\-mations of Virasoro algebra and conformal dimensions, Phys. Lett. B,  \textbf{262} (1), 32-38 (1991)
\bibitem{ChaiKuLuk}
Chaichian, M., Kulish, P., Lukierski, J.: $q$-Deformed Jacobi identity, $q$-oscillators and $q$-deformed infinite-dimensional algebras, Phys. Lett. B,  \textbf{237}, 401-406 (1990)
\bibitem{ChaiPopPres}
Chaichian, M., Popowicz, Z., Pre\v{s}najder, P.: $q$-Virasoro algebra and its relation to the $q$-deformed KdV system, Phys. Lett. B, \textbf{249}, 63-65 (1990)
\bibitem{CurtrZachos1}
Curtright, T. L., Zachos, C. K.: Deforming maps for quantum algebras, Phys. Lett. B, \textbf{243}, 237-244 (1990)
\bibitem{DamKu}
Damaskinsky, E. V., Kulish, P. P.: Deformed oscillators and their applications, Zap. Nauch. Semin. LOMI, \textbf{189}, 37-74 (1991) (in Russian) [Engl. tr. in J. Sov. Math., \textbf{62}, 2963-2986 (1992)
\bibitem{DaskaloyannisGendefVir}
Daskaloyannis, C.: Generalized deformed Virasoro algebras, Modern Phys. Lett. A, \textbf{7}(9), 809-816 (1992)
\bibitem{ElchingerLundMakhSilv:BracktausigmaderivWittVir}
Elchinger, O., Lundeng{\aa}rd, K., Makhlouf, A., Silvestrov, S. D.:
Brackets with $(\tau,\sigma)$-derivations and $(p,q)$-deformations of Witt and Virasoro algebras, Forum Math. \textbf{28}(4), 657-673 (2016)
\bibitem{Fe} Fechner, W.: Inequalities conected with averaging operators,  Indag. Math. {\bf 24}, 305-312 (2013)
\bibitem{GuanChenSun:HomLieSuperalgebras}
Guan, B., Chen, L., Sun, B.: On Hom-Lie superalgebras, Adv. Appl. Clifford Algebras, \textbf{29}(16) (2019)
\bibitem{HartwigLarssonSilvestrov:defLiealgsderiv}
Hartwig, J. T., Larsson, D., Silvestrov, S. D.:
Deformations of Lie algebras using $\sigma$-derivations, J. Algebra, \textbf{295}(2),  314-361 (2006) (Preprints in Mathematical Sciences 2003:32, LUTFMA-5036-2003, Centre for Mathematical Sciences, Lund University, 52 pp. (2003))
\bibitem{Hu}
Hu, N.: $q$-Witt algebras, $q$-Lie algebras, $q$-holomorph structure and representations,  Algebra Colloq. \textbf{6}(1), 51-70 (1999)
\bibitem{KF} Kamp\'e de F\'eriet, J.: L'etat actuel du probl\'eme de la turbulaence (I and II), La Sci. A\'erienne {\bf 3}, 9-34 (1934), {\bf 4} 12-52 (1935)
\bibitem{Kassel92}
Kassel, C.: Cyclic homology of differential operators, the Virasoro algebra and a $q$-analogue, Comm. Math. Phys. \textbf{146}(2), 343-356 (1992)
\bibitem{Kel} Kelley, J.~L.:  Averging operators on $C_{\infty}(X)$,  Illinois J. Math. {\bf 2}, (1958) 214-223.
\bibitem{KitouniMakhloufSilvestrov}
Kitouni, A., Makhlouf, A., Silvestrov, S.: On $(n+1)$-Hom-Lie algebras induced by $n$-Hom-Lie algebras, Georgian Math. J. \textbf{23}(1), 75-95 (2016)
\bibitem{kms:solvnilpnhomlie2020}
Kitouni, A., Makhlouf, A., Silvestrov, S.: On solvability and nilpotency for $n$-Hom-Lie algebras and $(n+1)$-Hom-Lie algebras induced by $n$-Hom-Lie algebras, In: Silvestrov, S., Malyarenko, A., Rancic, M. (Eds.), Algebraic Structures and Applications,
Springer Proceedings in Mathematics and Statistics, \textbf{317}, Springer, Ch  6, 127-157 (2020)
\bibitem{LarssonSigSilvJGLTA2008:QuasiLiedefFttN}
Larsson, D., Sigurdsson, G., Silvestrov, S. D.: Quasi-Lie deformations on the algebra $\mathbb{F}[t]/(t^N)$, J. Gen. Lie Theory Appl. \textbf{2}(3), 201-205 (2008)
\bibitem{LarssonSilvJA2005:QuasiHomLieCentExt2cocyid}
Larsson, D., Silvestrov, S. D.: Quasi-Hom-Lie algebras, central extensions and $2$-cocycle-like identities, J. Algebra \textbf{288}, 321-344 (2005) (Preprints in Mathematical Sciences 2004:3, LUTFMA-5038-2004, Centre for Mathematical Sciences, Lund University (2004))
\bibitem{LarssonSilv2005:QuasiLieAlg}
Larsson, D., Silvestrov, S. D.: Quasi-Lie algebras, In: Fuchs, J., Mickelsson, J., Rozenblioum, G., Stolin, A., Westerberg, A. (Eds.),
Noncommutative Geometry and Representation Theory in Mathematical Physics, Contemp. Math. \textbf{391}, Amer. Math. Soc., Providence, RI, 241-248 (2005) (Preprints in Mathematical Sciences 2004:30, LUTFMA-5049-2004, Centre for Mathematical Sciences, Lund University (2004))
\bibitem{LarssonSilv:GradedquasiLiealg}
Larsson, D., Silvestrov, S. D.: Graded quasi-Lie agebras, Czechoslovak J. Phys. \textbf{55}, 1473-1478 (2005)
\bibitem{LarssonSilv:QuasidefSl2}
Larsson, D., Silvestrov, S. D.: Quasi-deformations of $sl_2(\mathbb{F})$ using twisted derivations, Comm. Algebra, \textbf{35}, 4303-4318 (2007)
\bibitem{LarssonSilvestrovGLTMPBSpr2009:GenNComplTwistDer}
Larsson, D., Silvestrov, S. D.: On generalized $N$-complexes comming from twisted derivations, In: Silvestrov, S., Paal, E., Abramov, V., Stolin, A. (Eds.),
Generalized Lie Theory in Mathematics, Physics and Beyond, Springer-Verlag, Ch. 7, 81-88 (2009)
\bibitem{LiuKQuantumCentExt}
Liu, K. Q.: Quantum central extensions, C. R. Math. Rep. Acad. Sci. Canada \textbf{13}(4), 135-140 (1991)
\bibitem{LiuKQCharQuantWittAlg}
Liu, K. Q.: Characterizations of the quantum Witt algebra, Lett. Math. Phys. \textbf{24}(4), 257-265 (1992)
\bibitem{LiuKQPhDthesis}
Liu, K. Q.: The quantum Witt algebra and quantization of some modules over Witt algebra, PhD Thesis, Department of Mathematics, University of Alberta, Edmonton, Canada (1992)
\bibitem{Lo} Loday, J.-L.:  Dialgebras, In: Loday, J.-L., Frabetti, A., Chapoton, Goichot, F. (Eds.), Dialgebras and related operads,  Lecture Notes in Math. {\bf 1763}, 7-66 (2002)
\bibitem{MabroukNcibSilvestrov2020:GenDerRotaBaxterOpsnaryHomNambuSuperalgs}
Mabrouk, S., Ncib, O., Silvestrov, S.: Generalized derivations and Rota-Baxter operators of $n$-ary Hom-Nambu superalgebras, arXiv:2003.01080[math.QA]
\bibitem{Makhlouf2010:ParadigmnonassHomalgHomsuper}
Makhlouf, A.: Paradigm of nonassociative Hom-algebras and Hom-superalgebras,
Proceedings of Jordan Structures in Algebra and Analysis Meeting, 145-177 (2010)
\bibitem{ms:homstructure}
Makhlouf, A., Silvestrov, S. D.: Hom-algebra structures, J. Gen. Lie Theory Appl. \textbf{2}(2), 51-64 (2008)
(Preprints in Mathematical Sciences  2006:10, LUTFMA-5074-2006, Centre for Mathematical Sciences, Lund University (2006))
\bibitem{MakhSil:HomHopf}
Makhlouf, A., Silvestrov, S.:
Hom-Lie admissible Hom-coalgebras and Hom-Hopf algebras, In: Silvestrov, S., Paal, E., Abramov, V., Stolin, A. (Eds.),
Generalized Lie Theory in Mathematics, Physics and Beyond, Springer-Verlag, Berlin, Heidelberg, Ch. 17, 189-206 (2009) (Preprints in Mathematical Sciences 2007:25, LUTFMA-5091-2007, Centre for Mathematical Sciences, Lund Universty (2007). arXiv:0709.2413 [math.RA])
\bibitem{MakhSilv:HomDeform}
Makhlouf, A., Silvestrov, S.: Notes on $1$-parameter formal deformations of Hom-associative and Hom-Lie algebras, Forum Math. \textbf{22}(4), 715-739 (2010)
(Preprints in Mathematical Sciences 2007:31, LUTFMA-5095-2007, Centre for Mathematical Sciences, Lund University (2007). arXiv:0712.3130v1 [math.RA])
\bibitem{MakhSilv:HomAlgHomCoalg}
Makhlouf, A., Silvestrov, S. D.:
Hom-algebras and Hom-coalgebras, J. Algebra Appl. \textbf{9}(4), 553-589 (2010) (Preprints in Mathematical Sciences 2008:19, LUTFMA-5103-2008, Centre for Mathematical Sciences, Lund University (2008). arXiv:0811.0400[math.RA])
\bibitem{MandalMishra:HomGerstenhaberHomLiealgebroids}
Mandal, A., Mishra, S. K.: On Hom-Gerstenhaber algebras, and Hom-Lie algebroids, J. Geom. Phys. \textbf{133}, 287-302 (2018)
\bibitem{Mil} Miller, J.~B.: Averaging and Reynolds operators on Banach algebra I, Representation by derivation and antiderivations,  J. Math. ANAL. APPL. {\bf 14}, 527-548 (1966)
\bibitem{MishraSilvestrov:SpringerAAS2020HomGerstenhalgsHomLiealgds}
Mishra, S. K., Silvestrov, S.:  A review on Hom-Gerstenhaber algebras and Hom-Lie algebroids, In: Silvestrov S., Malyarenko A., Ran\u{c}i\'{c}, M. (Eds.), Algebraic Structures and Applications, Springer Proceedings in Mathematics and Statistics, \textbf{317}, Springer, Ch. 11, 285-315 (2020)
\bibitem{STC} Moy, S.T. C.: Characterizations of conditional expectation as a
transformation on function spaces, Pacific J. Math. {\bf 4}, 47-63 (1954)
\bibitem{PBGN} Pei, J.,  Bai, C., Guo, L., Ni, X.: Replicating of binary operads, Koszul duality, Manin products and averaging operators, arXiv:1212.0177v2.

\bibitem{PBGN2}  Pei, J.,  Bai, C., Guo, L., Ni, X.: Disuccessors and duplicators of operads, Manin products and operators, In: Symmetries and Groups in Contemporary Physics, Nankai Series in Pure, Applied Mathematics and Theoretical Physics, \textbf{11} 191-196 (2013).

\bibitem{RichardSilvJA2008:quasiLiesigderCtpm1}
Richard, L., Silvestrov, S. D.: Quasi-Lie structure of $\sigma$-derivations of $\mathbb{C}[t^{\pm1}]$, J. Algebra, \textbf{319}(3), 1285-1304 (2008)

\bibitem{RichardSilvestrovGLTMPBSpr2009:QuasiLieHomLiesigmaderiv}
Richard, L., Silvestrov, S.:
A Note on Quasi-Lie and Hom-Lie structures of $\sigma$-derivations of $\mathbb{C}[z_1^{\pm1},\dots,z_n^{\pm1}]$, In: Silvestrov, S., Paal, E., Abramov, V., Stolin, A. (Eds.), Generalized Lie Theory in Mathematics, Physics and Beyond, Springer-Verlag, Berlin, Heidelberg, Ch. 22, 257-262 (2009)
\bibitem{R1} Rota, G.-C.:  On the representation of averaging operator,  Rendiconti del Seminario Matematico della Universita di Padova, \textbf{30}, 52-64 (1960)
\bibitem{Sheng:homrep}
Sheng, Y.: Representations of Hom-Lie algebras, Algebr. Reprensent. Theory, \textbf{15}, 1081-1098 (2012)
\bibitem{ShengBai2014:homLiebialg}
Sheng, Y., Bai, C.: A new approach to Hom-Lie bialgebras, J. Algebra, \textbf{399}, 232-250 (2014)
\bibitem{ShengChen2013:HomLie2algebras}
Sheng, Y., Chen, D.: Hom-Lie $2$-algebras, J. Algebra \textbf{376}, 174-195 (2013)
\bibitem{ShengXiong:LMLA2015:OnHomLiealg}
Sheng, Y., Xiong Z.: On Hom-Lie algebras, Linear Multilinear Algebra, \textbf{63}(12), 2379-2395 (2015)
\bibitem{SigSilv:CzechJP2006:GradedquasiLiealgWitt}
Sigurdsson, G., Silvestrov, S.: Graded quasi-Lie algebras of Witt type, Czechoslovak J. Phys. \textbf{56}, 1287-1291 (2006)
\bibitem{SigSilv:GLTbdSpringer2009}
Sigurdsson, G., Silvestrov, S.: Lie color and Hom-Lie algebras of Witt type and their central extensions, In: Silvestrov, S., Paal, E., Abramov, V., Stolin, A. (Eds.), Generalized Lie Theory in Mathematics, Physics and Beyond, Springer-Verlag, Berlin, Heidelberg, Ch. 21, 247-255 (2009)
\bibitem{SilvestrovParadigmQLieQhomLie2007}
Silvestrov, S.: Paradigm of quasi-Lie and quasi-Hom-Lie algebras and quasi-defor\-mations, In: New techniques in Hopf algebras and graded ring theory, K. Vlaam. Acad. Belgie Wet. Kunsten (KVAB), Brussels, 165-177 (2007)
\bibitem{Yau:EnvLieAlg}
Yau, D.: Enveloping algebras of Hom-Lie algebras, J. Gen. Lie Theory Appl. \textbf{2}(2), 95-108 (2008)
\bibitem{Yau2009:HomYangBaxterHomLiequasitring} Yau, D.: Hom-Yang-Baxter equation, Hom-Lie algebras, and quasi-triangular bialgebras, J. Phys. A, \textbf{42}, 165202 (2009)
\bibitem{Yau:HomolHom}
Yau, D.: Hom-algebras and homology, J. Lie Theory, \textbf{19}(2), 409-421 (2009)
\bibitem{Yau:HomBial} Yau, D.: Hom-bialgebras and comodule algebras, Int. Electron. J. Algebra, \textbf{8}, 45-64 (2010)
\bibitem{Yuan2012:HomLiecoloralgstr} Yuan, L.: Hom-Lie color algebra structures, Comm. Algebra, \textbf{40}, 575-592 (2012)

\bibitem{YY} Yuan, L. M., You, H.: Low dimensional cohomology of Hom-Lie algebra and $q$-deformed $W(2,2)$ algebra, Acta Mathematica Sinica, English Series, {\bf 30}(6) (2014), 1073--1082.

\bibitem{GK2} Gubarev, V.~Yu., Kolesnikov, P.~S.:  On embedding of dendriform algebras into Rota-Baxter algebras, Cent. Eur. Jour. Math {\bf 11}, 226-245 (2013)

\bibitem{ZhouChenMa:GenDerHomLiesuper}
Zhou, J., Chen, L., Ma, Y.: Generalized derivations of Hom-Lie superalgebras, Acta Math. Sinica (Chin. Ser.) \textbf{58}, 3737-3751 (2014)
\bibitem{ZhouChenMa:GenDerLieTripSyst}
Zhou, J., Chen, L., Ma, Y.: Generalized derivations of Lie triple systems,
Open Math., \textbf{14}(1), 260-271 (2016) (arXiv:1412.7804 (2014))
\bibitem{ZhouNiuChen:GhomDerivation}
Zhou, J., Niu, Y. J., Chen, L. Y.: Generalized derivations of Hom-Lie algebras,
Acta Mathematica Sinica, Chinese Series, \textbf{58}(4), 551-558 (2015)
\bibitem{ZhouZhaoZhang:GenDerHomLeibnizalg}
Zhou, J., Zhao, X., Zhang, Y.: Generalized derivations of Hom-Leibniz algebras, J. Jilin University (Science Edition), \textbf{55}(02), 195-200 (2017)	

\end{thebibliography}
\end{document}